\documentclass{amsart}

\title[Closed convex hulls of unitary orbits]{Closed convex hulls of unitary orbits in certain simple real rank zero C$^*$-algebras} 

\usepackage{microtype}
\usepackage{hyperref}
         
\usepackage{amsmath}             
\usepackage{amsfonts}             
\usepackage{amsthm}               
\usepackage{amsbsy}
\usepackage{amssymb}
\usepackage{amsthm}
\usepackage[initials, nobysame]{amsrefs}

\usepackage{enumerate}

\newtheorem{thm}{Theorem}[section]
\newtheorem{lem}[thm]{Lemma}

\newtheorem{cor}[thm]{Corollary}

\theoremstyle{definition}
\newtheorem{defn}[thm]{Definition}

\newtheorem{rem}[thm]{Remark}
\newtheorem{exam}[thm]{Example}

\newcommand{\bC}{{\mathbb{C}}}
\newcommand{\bN}{{\mathbb{N}}}
\newcommand{\bR}{{\mathbb{R}}}

\newcommand{\fM}{{\mathfrak{M}}}

\newcommand{\A}{{\mathcal{A}}}
\newcommand{\B}{{\mathcal{B}}}
\newcommand{\C}{{\mathcal{C}}}
\newcommand{\D}{{\mathcal{D}}}
\newcommand{\E}{{\mathcal{E}}}
\newcommand{\F}{{\mathcal{F}}}

\newcommand{\M}{{\mathcal{M}}}

\newcommand{\U}{{\mathcal{U}}}

\newcommand{\qand}{\quad\text{and}\quad}
\newcommand{\qqand}{\qquad\text{and}\qquad}

\newcommand{\Tr}{\mathrm{Tr}}
\newcommand{\tr}{\mathrm{tr}}
\newcommand{\diag}{\mathrm{diag}}
\newcommand{\sa}{\mathrm{sa}}
\newcommand{\conv}{\mathrm{conv}}
\newcommand{\cconv}{\overline{\mathrm{conv}}}
\newcommand{\dist}{\mathrm{dist}}
\renewcommand{\th}{\mathrm{th}}
\renewcommand{\sp}{\mathrm{sp}}

\numberwithin{equation}{section}

\begin{document}

\author{P. W. Ng}
\address{Department of Mathematics, University of Louisiana at Lafayette, Lafayette, Louisiana, USA, 70504-1010}
\email{png@louisiana.edu}

\author{P. Skoufranis}
\address{Department of Mathematics, Texas A\&M University, College Station, Texas, USA, 77843-3368}
\email{pskoufra@math.tamu.edu}

\subjclass[2010]{46L05}
\date{\today}
\keywords{Convex Hull of Unitary Orbits; Real Rank Zero C$^*$-Algebras; Simple; Eigenvalue Functions; Majorization}

\begin{abstract}
In this paper, we characterize the closures of convex hulls of unitary orbits of self-adjoint operators in unital, separable, simple C$^*$-algebras with non-trivial tracial simplex, real rank zero, stable rank one, and strict comparison of projections with respect to tracial states.  In addition, an upper bound for the number of unitary conjugates in a convex combination needed to approximate a self-adjoint are obtained.
\end{abstract}

\maketitle

\section{Introduction}

Approximate unitary equivalence of objects is an essential notion to the theory of operator algebras.  For example, approximately unitary equivalence of homomorphisms produces K-theoretic data and the notion of when two self-adjoint operators are approximately unitarily equivalent inside a  unital C$^*$-algebra has been important concept for decades.  

A problem that has received less attention is to characterize the closure of the convex hull of the unitary orbit of an operator inside a unital C$^*$-algebra.    From the operator-theoretic standpoint, unitary conjugates of a self-adjoint operator $a$ inside a unital C$^*$-algebra are, for all intents and purposes, the same operator as $a$ whereas elements in the convex hull may be viewed as `averages' of different copies of $a$. 
Thus characterizing the closure of the unitary orbit of a self-adjoint operator (an equivalent method for determining approximate unitary equivalence) need not characterize its convex hull as it is not clear that spectral data is well-behaved with respect to this averaging process.

One reason for pondering this problem is that given a self-adjoint operator $a$ in a C$^*$-algebra $\A$ we often take limits of elements of $\A$ of the form
\[
\frac{1}{n}\sum^n_{k=1} u_k^* a u_k
\]
where $u_k \in \A$ are unitary operators.  Consequently, the closure of the convex hull of the unitary orbit of $a$ characterizes which operators may be obtained.  In addition, there are many other problems in operator algebras where closed convex hulls of unitary orbits plays a vital role including the classical theorem of Schur and Horn (see \cites{S1923, H1954}) which characterizes the possible diagonal $n$-tuples of a self-adjoint matrix based on its eigenvalues, the  extension of the Schur-Horn Theorem to other C$^*$-algebras (see \cites{AM2007, BJ2015, BJ2014, AK2006, R2012,  KS2015-1, KS2015-2, MR2014} to name but a few), and to generalized numerical ranges of operators (see \cites{GS1977, P1980, DS2015} to name but a few).

All of the above applications have mainly been pursued in the von Neumann algebra setting.  Using a notion of majorization first defined for $L_1$-functions by Hardy, Littlewood, and P\'{o}lya (see \cites{HLP1929, HLP1952}) and extended to self-adjoint operators in von Neumann algebras with faithful tracial states by Kamei in \cite{K1983} (also see \cites{AM2008, AK2006, F1982, FK1986, H1987, K1984, K1985, P1985}), it is possible to characterize the convex hull of the unitary orbit of self-adjoint operators in tracial von Neumann algebras (see \cites{HN1991}).  Recently, the second author extended these results to certain unital C$^*$-algebras with real rank zero (see \cite{S2015}) whereas the first author provided a characterization for certain C$^*$-algebra with no non-trivial projections (see \cite{N2016}).

The goal of this paper is to characterize closed convex hulls of unitary orbits of self-adjoint operators in unital, separable, simple C$^*$-algebras with non-trivial tracial simplex, real rank zero, strict comparison of projections, and cancellation of projections.  In addition to this introduction, this paper contains five additional sections structured as follows.

In Section 2 many preliminaries on eigenvalue functions will be analyzed for arbitrary unital C$^*$-algebras.  In particular, the notion and properties of eigenvalue functions  may be directly imported  from von Neumann algebra theory.  This leads to the notion of majorization of self-adjoint operators in C$^*$-algebra.  The goal of this paper is to connect the notion of majorization for self-adjoint operators to the closed convex hull of unitary orbits.

In Section 3 a technical result on comparing eigenvalue functions in tracially approximately finite dimensional C$^*$-algebras (TAF C$^*$-algebras) will be developed.  The reason for analyzing TAF C$^*$-algebras is that every C$^*$-algebra we wish to study contains a TAF C$^*$-algebra with identical $K$-theoretical data.  As such, the single operator theory problem under consideration in this paper may then be solved by placing the operators inside a TAF C$^*$-algebra.  TAF C$^*$-algebras are particularly suited for this problem as, up to a portion of the algebra of arbitrarily small trace, TAF C$^*$-algebras look like finite dimensional C$^*$-algebras where a solution to our problem exists.  Consequently, Section 3 shows that if one self-adjoint operator almost majorizes another, then the same holds for the finite dimensional approximates.

In Section 4 many additional approximation results relating to convex hulls of unitary orbits will be developed using matricial results.  In particular, we demonstrate that if a self-adjoint matrix $B$ almost majorizes a self-adjoint matrix $A$, then $A$ is almost in the convex hull of the unitary orbit of $B$.  Subsequently, we also develop a method for handling the portion of the TAF C$^*$-algebra with arbitrarily small trace.

Finally, Section 5 combines the results of the previous sections along with some simple approximations to prove the main result of this paper, Theorem \ref{thm:main}.   Consequently, for the C$^*$-algebras under consideration in this paper, Theorem \ref{thm:main} classifies when one self-adjoint operator is in the closed convex hull of the unitary orbit of another self-adjoint operator using the notion of majorization from Section 2.  Furthermore, analyzing the proof of Theorem \ref{thm:main} yields a method for constructing the convex combination of unitary conjugates in such C$^*$-algebras provided one knows how to construct convex combination of unitary conjugates in matrix algebras (and such a method exists by \cites{A1989,B1946} and Lemma \ref{lem:a-convex-combination-of-2}).

It is interesting to note the existence of an AF C$^*$-algebra $\A$ with strict comparison of projections with respect to a unique faithful tracial state $\tau$ such that there exist projections $p,q \in \A$ where $\tau(p) = \tau(q)$ and yet $p$ and $q$ are not approximately unitarily equivalent (see \cite{B1988}).  However, as $\tau(p) = \tau(q)$ implies that $p$ and $q$ have the same eigenvalue functions with respect to $\tau$, Theorem \ref{thm:main} implies that $p$ is in the closed convex hull of the unitary orbit of $q$.  Consequently the extreme points of the convex hull of the unitary orbit of $p$ is not the unitary orbit of $p$.

To conclude the paper, Section 6 analyzes how many unitaries are required in a convex combination of a self-adjoint operator to obtain a given self-adjoint operator.  It is known that if $p$ is a rank one projection in the $n \times n$ matrices, then $\frac{1}{n} I_n$ is a convex combination of unitary conjugates of $p$.  However, a simple rank argument implies that at least $n$ unitaries are required in a convex combination of unitary conjugates of $p$ to obtain $\frac{1}{n} I_n$.  Consequently, as we desire a uniform bound over all $n$, this question will not be well-defined unless we restrict our attention to approximating another operator instead of precisely obtaining it.  By a matricial result obtained in Section 6 and under the hypotheses of this paper,  Theorem \ref{thm:main} may be extended to show that if $a$ is in the closed convex hull of the unitary orbit of $b$, then $a$ can be approximated within $O(\frac{1}{n})$ by a convex combination of $O(2^{\frac{n^3}{3}})$ unitary conjugates of $b$.

\section{Preliminaries on Eigenvalue Functions}

In this section, the preliminary structures needed to discuss closed convex hulls of unitary orbits will be developed.  Throughout this paper, $\A$ will denoted a unital C$^*$-algebra, $\A_{\sa}$ will denote the set of self-adjoint elements of $\A$, $\A_+$ will denote the set of positive elements of $\A$, $\U(\A)$ will denote the set of unitary elements of $\A$, $T(\A)$ will denote the set of tracial states on $\A$ (which will be assumed to be non-empty), and $\partial_{\mathrm{ext}} T(\A)$ will denote the extreme boundary of $T(\A)$.  Given $a \in \A_{\sa}$, the spectrum of $a$ will be denote $\sp(a)$ and the unitary orbit of $a$ in $\A$ will be denoted
\[
\U(a) = \{u^*au \, \mid \, u \in \U(\A)\}.
\]
Furthermore, $\conv(\U(a))$ will be used to denote the convex hull of $\U(a)$, and $\cconv(\U(a))$ will be used to denote the norm closure of $\conv(\U(a))$.  Finally, given $n \in \bN$, $\M_n$ will denote the C$^*$-algebra of $n \times n$ matrices, $\tr$ will denote the normalized tracial state on $\M_n$, and $\Tr$ will denoted the unnormalized trace on $\M_n$.

Given $a \in \A_{\sa}$ and $\tau \in T(\A)$, notice that the isomorphism $C^*(a, 1) \cong C(\sp(a))$ induces a Borel probability measure $\mu_{\tau, a}$ on $\sp(a)$ by the Riesz Representation Theorem. Using $\mu_{\tau, a}$, we can defined the following concept which origins in the work of Murray and von Neumann.

\begin{defn}[\cites{MN1936, F1982, FK1986}]
Given a unital C$^*$-algebra $\A$, $a \in \A_{\sa}$, and $\tau \in T(\A)$, the \emph{eigenvalue function of} $a$ \emph{associated with} $\tau$, denoted $\lambda^\tau_a$, is defined for $s \in [0,1)$ by
\[
\lambda^\tau_a(s) := \inf\{t \in \bR \, \mid \, \mu_{\tau, a}((t,\infty)) \leq s\}.
\]
\end{defn}

It is elementary to see that $\lambda^\tau_a$ depends only on the values of $\{\tau(a^n) \, \mid \, n \in \bN\}$.  One particular example will be of great use throughout the paper is as follows.

\begin{exam}
\label{exam:finite-spectrum-eigenvalue-functions}
Let $\A$ be a unital C$^*$-algebra and let $\tau \in T(\A)$.  Let $\{\alpha_k\}^n_{k=1} \subseteq \bR$ be such that $\alpha_k > \alpha_{k+1}$ for all $k$ and let $\{p_k\}^n_{k=1} \subseteq \A$ be a collection of pairwise orthogonal projections such that $\sum^n_{k=1} p_k = 1_\A$.  If $a = \sum^n_{k=1} \alpha_k p_k$ then it is elementary to see that $\mu_{\tau, a} = \sum^n_{k=1} \tau(p_k) \delta_{\alpha_k}$ where $\delta_x$ denotes the point-mass measure at $x$.  Consequently we see that if $s_k = \sum^k_{j=1} \tau(p_j)$ for each $k$ then $\lambda^\tau_a(s) = \alpha_k$ for all $s \in [s_{k-1}, s_k)$.
\end{exam}

Many properties of eigenvalue functions may be recovered directly via von Neumann algebra theory.    For a unital C$^*$-algebra $\A$ and $\tau \in T(\A)$, let $\pi_\tau : \A \to \B(L_2(\A, \tau))$ denote the GNS representation of $\A$ with respect to $\tau$.  Recall this implies there exists a unit vector $\xi \in L_2(\A, \tau)$ such that $\xi$ is a cyclic vector for $\pi$ and
\[
\tau(a) = \langle \pi_{\tau}(a) \xi, \xi\rangle
\]  
for all $a \in \A$.   Let $\fM$ be the von Neumann algebra generated by $\pi_\tau(\A)$ (i.e. $\fM = \overline{\pi_\tau(\A)}^{\text{WOT}}$) and define $\tau_0 : \fM \to \bC$ via
\[
\tau_0(x) = \langle x \xi, \xi\rangle
\]
for all $x \in \fM$.  Using $\pi_\tau$ and the weak-operator topology density of $\pi_\tau(\A)$ in $\fM$, it is elementary to show that $\tau_0$ is a tracial state on $\fM$.  Furthermore $\tau_0$ is faithful on $\fM$.  To see this, suppose $x \in \fM_+$ is such that $\tau_0(x) = 0$.  By the Cauchy-Schwarz inequality, this implies $\tau_0(x^\frac{1}{2}\pi_\tau(a)) = 0$ for all $a \in \A$.  Consequently, for all $a \in \A_+$,
\[
0 = \tau_0\left(x^\frac{1}{2}\pi_\tau(a^2)\right) = \tau_0\left(\pi_\tau\left(a\right) x^\frac{1}{2} \pi_\tau\left(a\right) \right) = \langle \pi_\tau\left(a\right) x^\frac{1}{2} \pi_\tau\left(a\right) \xi, \xi \rangle = \left\| x^\frac{1}{4} \pi_\tau\left(a\right) \xi\right\|^2.
\]
Hence $x^\frac{1}{4} \pi_\tau(a) \xi = 0$ for all $a \in \A_+$.  Therefore, as $\pi_\tau(\A)\xi$ is dense in $L_2(\A, \tau)$, we obtain that $x^\frac{1}{4} = 0$ and hence $x = 0$.  Hence $\tau_0$ is faithful on $\fM$.

Since $\lambda^\tau_a = \lambda^{\tau_0}_{\pi_{\tau}(a)}$ as $\tau(a^n) = \tau_0(\pi_\tau(a)^n)$ for all $n \in \bN$, the following is easily imported from von Neumann algebra theory.

\begin{thm}
\label{thm:eigenvalue-function-properties}
Let $\A$ be a unital C$^*$-algebra, let $\tau \in T(\A)$, and let $a, b \in \A_{\sa}$.  Then the following are true:
\begin{enumerate}
\item The map $s \mapsto \lambda^\tau_a(s)$ is non-increasing and right continuous.
\item $\lim_{s \nearrow 1} \lambda^\tau_a(s) \geq \inf\{t \, \mid \, t \in \sp(a)\}$ and $\lim_{s \searrow 0} \lambda^\tau_a(s) \leq \sup\{t \, \mid \, t \in \sp(a)\}$, with equalities if $\tau$ is faithful.\label{EFspectrallimits}
\item If $a \leq b$ then $\lambda^{\tau}_a(s) \leq \lambda^{\tau}_b(s)$
for all $s \in [0,1)$. \label{EFOrder}
\item For all $\alpha \in [0, \infty)$, 
$\lambda^{\tau}_{\alpha a}(s) = \alpha \lambda^{\tau}_a(s)$ for all $s \in [0,1)$. \label{EFDilation}
\item For all $\alpha \in \mathbb{R}$,
$\lambda^{\tau}_{a + \alpha 1}(S) = \lambda^{\tau}_a(s) + \alpha$ for all $s \in [0,1)$. 
\label{EFTranslation} 
\item $|\lambda^{\tau}_a (s) - \lambda^{\tau}_b(s) | \leq 
\left\| a - b \right\|$ for all $s \in [0,1)$. \label{EFNorm}  
\end{enumerate} 
\label{thm:EFunctionProperties}
\end{thm}
\begin{proof}
If $\pi_\tau : \A \to \B(L_2(\A, \tau))$ is the GNS representation of $\A$ with respect to $\tau$ and $\tau_0$ is the faithful tracial state induced by $\tau$ on the von Neumann algebra generated by $\pi_\tau(\A)$, then the result holds when $\lambda^\tau_a$ and $\lambda^\tau_b$ are replaced with $\lambda^{\tau_0}_{\pi_{\tau}(a)}$ and $\lambda^{\tau_0}_{\pi_{\tau}(b)}$ (with equalities in (\ref{EFspectrallimits})) by \cites{F1982, FK1986, P1985}.  Consequently, the result follows.
\end{proof}

Using eigenvalue functions, it is possible to define a notion of majorization with respect to a tracial state.
\begin{defn}
\label{defn:majorization}
Let $\A$ be a unital C$^*$-algebra, let $\tau \in T(\A)$, and let 
$a,b \in \A_{sa}$.  It is said that $a$ is \emph{majorized by} $b$ \emph{with respect to} $\tau$, denoted $a \prec_\tau b$, if 
\begin{enumerate}
\item $\int^t_0 \lambda^\tau_a(s) \, ds \leq \int^t_0 \lambda^\tau_b(s) \, ds$ for all $t \in [0,1]$, and
\item $\int^1_0 \lambda^\tau_a(s) \, ds = \int^1_0 \lambda^\tau_b(s) \, ds$.
\end{enumerate}
\end{defn}

The following provides the connection between majorization of eigenvalue functions and closed convex hulls of unitary orbits of self-adjoint operators in tracial von Neumann algebra factors.
\begin{thm}[see \cites{AK2006, AM2008, A1989, B1946, H1987, HN1991, K1983, K1984, K1985}]
\label{thm:majorization-in-factors}
Let $\fM$ be a von Neumann algebra with a faithful tracial state $\tau$ and let $a,b \in \fM_+$.  Then the following are equivalent:
\begin{enumerate}
\item $a \prec_\tau b$.
\item $\tau((a- \alpha 1_\fM)_+) \leq \tau((b - \alpha 1_\fM)_+)$ for all $r > 0$ and $\tau(a) = \tau(b)$.
\item $\tau(f(a)) \leq \tau(f(b))$ for every continuous convex function $f : \bR \to \bR$. 
\end{enumerate}
Furthermore, if $\fM$ is a factor, then for all self-adjoint $a, b \in \fM_{\sa}$, $a \prec_\tau b$ is equivalent to:
\begin{enumerate}
\setcounter{enumi}{3}
\item $a \in \cconv(\U(b))$. 
\item $a \in \overline{\conv(\U(b))}^{w^*}$. 
\item There exists a unital, trace-preserving, positive map $\Phi : \fM \to \fM$ such that $\Phi(b) = a$.  
\item There exists a unital, trace-preserving, completely positive map $\Phi : \fM \to \fM$ such that $\Phi(b) = a$.  
\end{enumerate}
\end{thm}

The following describes an alternate way of viewing the quantities in the definition of majorization without the need of eigenvalue functions and will have a technical use later in the paper.

\begin{lem}
\label{lem:IntLambdaChar}
If $\A$ is a unital C$^*$-algebra, $\tau \in T(\A)$, $a \in \A_{\sa}$, and $t \in [0,1]$, then
\[
\int_0^t \lambda_a^{\tau}(s) \, ds = \sup\{ \tau(ac) \, \mid \, c \in \A, 0 \leq c \leq 1_\A,
\tau(c) = t \}.
\]
Hence
\[
\int_0^1 \lambda_a^{\tau}(s) \, ds = \tau(a).
\]
Furthermore, if $a \in \A_+$, then
\[
\int_0^t \lambda_a^{\tau}(s) \, ds = \sup\{ \tau(ac)\, \mid \, c \in \A, 0 \leq c \leq 1_\A,
\tau(c) \leq t \}.
\]
\end{lem}
\begin{proof}
Let $\pi_\tau : \A \to \B(L_2(\A, \tau))$ be the GNS representation of $\A$ with respect to $\tau$, let $\fM$ be the von Neumann algebra generated by $\pi_\tau(\A)$, and let $\tau_0$ be the faithful tracial state on $\fM$ induced by $\tau$.  Then
\begin{align*}
&\int_0^t \lambda_a^{\tau}(s) \, ds
\\ &= \int_0^t \lambda_{\pi_\tau(a)}^{\tau_0}(s) \, ds\\
&= \sup\{ \tau_0(\pi_\tau(a)x)\, \mid \, x \in \fM, 0 \leq x \leq 1_{\fM_0}, \tau(x) = t \} \quad \text{ by \cite{P1985}*{Theorem 3}} \\
&= \sup\{ \tau_0(\pi_\tau(a)\pi_\tau(c))\, \mid \, c \in \A, 0 \leq c \leq 1_\A, \tau(c) = t \} \quad \text{ by Kaplansky Density} \\
&= \sup\{ \tau(ac) \, \mid \, c \in \A, 0 \leq c \leq 1, \tau(c) = t \}.
\end{align*}
Furthermore, if $a \in \A_+$, then $\lambda^\tau_a(s) \geq 0$ for all $s \in [0,1)$ by part (\ref{EFOrder}) of Theorem \ref{thm:EFunctionProperties}.  Hence
\begin{align*}
\int_0^t \lambda_a^{\tau}(s) \, ds &= \sup\left\{ \left. \int_0^r \lambda_a^{\tau}(s) \, ds \, \right| \, 0 \leq r \leq t  \right\} \\
&= \sup\{ \tau(ac)\, \mid \, c \in \A, 0 \leq c \leq 1, \tau(c) \leq t \}. \qedhere
\end{align*}
\end{proof}

\section{An Asymptotic Argument in TAF C$^*$-Algebras}

In this section, specific inequalities relating eigenvalue functions will be discussed.  In particular, the goal of this section is to prove Lemma \ref{lem:AsymptoticArgument}, which enables one to transfer inequalities involving integrals of eigenvalue functions to matricial subalgebras in certain C$^*$-algebras.  The C$^*$-algebras under consideration in this section are the TAF C$^*$-algebras due to Lin:
\begin{defn}[see \cite{H2000}*{Definition 2.1, Proposition 3.8, and Theorem 3.4}, 
\cite{LinBook}*{Section 3.6}, \cite{LinTR0Classification}, and the references therein]
Let $\A$ be a unital, separable, simple, non-elementary C$^*$-algebra.  It is said that $\A$ is \emph{tracially approximately finite dimensional (TAF)}
if for every $a \in \A_+ \setminus \{ 0 \}$, 
for every finite subset $\F \subseteq \A$,  
and for every $\epsilon > 0$, there exists a finite dimensional
C$^*$-subalgebra $\D$ such that
\begin{enumerate}
\item $1_{\A} - 1_{\D}$ is Murray-von Neumann equivalent to a projection
in $\text{Her}(a) := \overline{a \A a}$,
\item $\left\| 1_{\D} x - x 1_{\D} \right\| <\epsilon$ for all $x \in \F$, and
\item for every $x \in \F$, there exists $y_x \in \D$ such that 
\[
\left\| x - (1_{\A} - 1_{\D})x(1_{\A} - 1_{\D}) - y_x \right\| < \epsilon.
\]
\end{enumerate} 
\label{defn:TAFAlgebra}
\end{defn}
\begin{rem}
TAF C$^*$-algebras are also called C$^*$-algebras with \emph{tracial rank zero} 
(see \cite{LinBook}*{Section 3.6}, \cite{LinTRank}*{Theorem 7.1}, \cite{LinTR0Classification},
 and the references therein).   Furthermore, by \cite{H2000}*{Theorem 3.4 and Theorem 3.6} (see also 
\cite{B1998}*{Theorem 6.8.5 and Corollary 6.9.2}), and   
by \cite{LinBook}*{Theorem 3.6.11 and Theorem 3.7.2},    
all unital, simple TAF C$^*$-algebras have real rank zero, stable rank one, 
and strict comparison of projections with respect to tracial states.
\end{rem}

\begin{rem}
\label{rem:TAFLargeDenominators} 
By \cite{H2000}*{Lemma 6.10}, the finite dimensional C$^*$-algebra $\D$ in Definition \ref{defn:TAFAlgebra} may always be chosen so that its simple summands have arbitrarily large rank; that is, for every $N \geq 1$, we can
choose $\D$ so that if $\D \cong \M_{n_1} \oplus \M_{n_2} \oplus \cdots
\oplus \M_{n_k}$ then $n_j \geq N$ for all $j$.  For the convenience of the reader, we quickly sketch this result.

Since unital, simple, TAF C$^*$-algebras have real rank zero, it suffices to prove that if $\A$ is a unital, simple, non-elementary C$^*$-algebra with real rank zero and $\D \subseteq \A$ is a finite-dimensional C$^*$-subalgebra, then for every $N \geq 1$ there exists 
a finite dimensional C$^*$-subalgebra $\D \subseteq \D_1 \subseteq \A$   
such that every simple summand of $\D_1$ has rank at least $N$. 

Since compressions of $\A$ preserve the hypotheses on $\A$, we may assume that $\D$ is a full matrix algebra.
Let $\{ e_{i,j} \}^n_{i,j=1}$ be a system of matrix units for $\D$.  By \cite{PereraRordam}*{Proposition 5.3} (see also \cite{ZhangRiesz})
there exists a finite dimensional C$^*$-algebra $\D_0$ with each simple summand
having rank at least $N$ and a unital *-embedding $\phi : \D_0 \to e_{1,1} \A e_{1,1}$.  If $\D_1$ is the C$^*$-algebra generated
by $\D$ and $\phi(\D_0)$, then $\D_1$ has the required desired property.
\end{rem}

In order to describe and prove Lemma \ref{lem:AsymptoticArgument}, it is necessary to fix some notation and perform some constructions.  Let $\A$ be a unital, separable, simple, non-elementary TAF C$^*$-algebra and let $\{ \F_n \}_{n=1}^{\infty}$ be an increasing sequence of finite 
subsets of the closed unit ball of $\A$ that are closed under adjoints such that 
$\bigcup_{n=1}^{\infty} \F_n$ is dense in the closed unit ball of $\A$.

By the definition and properties of TAF C$^*$-algebra stated previously, and 
by \cite{H2000}*{Proposition 2.4}, for all $n \geq 1$ there exists a finite dimensional C$^*$-algebra $\D_n \subseteq \A$ such that
\begin{enumerate}[(a)]
\item  $\tau(1_{\A} - 1_{\D_n}) < 1/n$ for all $\tau \in T(\A)$,
\item  $\left\| 1_{\D_n} x - x 1_{\D_n} \right\| < 1/n$ for all $x \in \F_n$, and \label{property-almost-commute}
\item for every $x \in \F_n$, there exists a $x_n \in \D_n$ such that 
\[
\left\| x -  (1_{\A} - 1_{\D_n})x(1_{\A} - 1_{\D_n}) - x_n \right\| < 1/n.
\]
\end{enumerate}
Since $\bigcup_{k=1}^{\infty} \F_k$ is dense in the closed unit ball of
$\A$, for all $x \in \A$
there exists a sequence $(x_n)_{n\geq 1}$ such that 
$x_n \in \D_n$ for all $n \geq 1$ and 
\[
\lim_{n \to \infty} \left\| x - (1_{\A} - 1_{\D_n})x (1_{\A} - 1_{\D_n}) - x_n \right\| = 0.
\]
Notice if $(\hat{x}_n)_{n \geq 1}$ is another sequence such that $\hat{x}_n \in \D_n$ for all $n\geq 1$ and
\[
\lim_{n \to \infty} \left\| x - (1_{\A} - 1_{\D_n})x (1_{\A} - 1_{\D_n}) - \hat{x}_n \right\|= 0,
\]
then $\lim_{n \to \infty} \left\| x_n - \hat{x}_n \right\|= 0$.   For the purposes of this section, given an element $x \in \A$, $(x_n)_{n\geq 1}$ will denote such a sequence.

If $y \in \A$ and $(y_n)_{n \geq 1}$ is such that $y_n \in \D_n$ for all $n\geq 1$ and
\[
\lim_{n \to \infty} \left\| y - (1_{\A} - 1_{\D_n})y (1_{\A} - 1_{\D_n}) - y_n \right\|= 0,
\]
then $(x_n + y_n)_{n\geq 1}$ and $(y_n^*)_{n\geq 1}$ are such sequences for $x+y$ and $y^*$ respectively.  To see that $(x_ny_n)_{n\geq 1}$ is such a sequence for $xy$, notice if $x,y \in \F_k$ for some $k$, then
\begin{align*}
&\lim_{n \to \infty} \left\| xy - (1_{\A} - 1_{\D_n})xy (1_{\A} - 1_{\D_n}) - x_ny_n \right\| \\
&= \lim_{n \to \infty} \left\| xy - xy (1_{\A} - 1_{\D_n}) - x_ny_n \right\| \\
&= \lim_{n \to \infty} \left\| x(y - y (1_{\A} - 1_{\D_n}) - y_n) + (x - x_n) 1_{\D_n} y_n \right\| \\
&= \lim_{n \to \infty} \left\| x(y - y (1_{\A} - 1_{\D_n}) - y_n) + (x 1_{\D_n}  - x_n) y_n \right\| \\
&= \lim_{n \to \infty} \left\| x(y - y (1_{\A} - 1_{\D_n}) - y_n) + (x - x (1_\A - 1_{\D_n})  - x_n) y_n \right\| \\
&= \lim_{n \to \infty} \| x(y - (1_{\A} - 1_{\D_n}) y (1_{\A} - 1_{\D_n}) - y_n) \\ & \qquad \qquad \qquad + (x - (1_{\A} - 1_{\D_n}) x (1_\A - 1_{\D_n})  - x_n) y_n \|.
\end{align*}
The result for general $x$ and $y$ then follows from the density of $\bigcup_{k=1}^{\infty} \F_k$ in the closed unit ball of $\A$.

Let $x \in \A_{\sa}$.   By the above, we may assume that each $x_n$ is self-adjoint.  Furthermore we see that if $f : \bR \to \bR$ is continuous, then 
\[
\lim_{n \to \infty} \| f(x_n) - f(x)_n \| = 0.
\]
Consequently, we may further assume that $\left\|x_n\right\| \leq \left\|x\right\|$ and $\sp(x_n) \subseteq \sp(x)$ for all $n\geq 1$.

Since each $\D_n$ is a finite dimensional C$^*$-algebra, we may write
\[
\D_n = \D_{n,1} \oplus \D_{n,2} \oplus \cdots \oplus \D_{n, m_n}
\]
where each $\D_{n, j}$ is a full matrix algebra.  For all $x \in \A$,  $n \geq 1$, and $1 \leq j \leq m_n$, let $x_{n,j} \in \D_{n,j}$ denote the compression of $x_n$ onto $\D_{n,j}$.

The following shows that certain inequalities involving eigenvalue functions pass to the finite dimensional approximates.

\begin{lem}
\label{lem:AsymptoticArgument}
Using the above construction and notation, we have the following:
\begin{enumerate}

\item Suppose that $a \in \A_\sa$ and $\alpha \in \mathbb{R}$ are such that 
\[
\tau(a) \geq \alpha
\]
for every $\tau \in T(\A)$.  For every $\epsilon > 0$ there exists an $N \geq 1$ such that
\[
\tau(a_n) + \epsilon \geq \alpha
\] 
for all $n \geq N$ and $\tau \in \partial_{\mathrm{ext}}T(\D_n)$.

\item Let $a,b \in \A_\sa$ have finite spectrum and let $\epsilon > 0$.  Suppose for all $\tau \in T(\A)$ and $t \in [0,1]$ that
\[
\int_0^t \lambda_{a}^{\tau}(s) \, ds \leq \epsilon t + \int_0^t 
\lambda_{b}^{\tau} (s) \, ds
\qand
\epsilon + \int_0^1 \lambda_{a}^{\tau}(s) \, ds > 
\int_0^1 \lambda_{b}^{\tau} (s) \, ds.
\]
Then for every $\epsilon' > \epsilon$ there exists an $N \geq 1$ such that 
for all $n \geq N$, for all $\tau \in \partial_{\mathrm{ext}}T(\D_n)$, and
for all $t \in [0,1]$, we have
\[
\int_0^t \lambda_{a_n}^{\tau}(s) \, ds \leq \epsilon' t + \int_0^t 
\lambda_{b_n}^{\tau} (s) \, ds
\qand
\epsilon' + \int_0^1 \lambda_{a_n}^{\tau}(s) \, ds > 
\int_0^1 \lambda_{b_n}^{\tau} (s) \, ds.
\] 
\end{enumerate}
\end{lem}

\begin{proof}
To see that (1) is true, suppose to the contrary that there exists an $\epsilon > 0$ and sequences
$( n_k )_{k\geq 1}$ and $( j_k )_{k\geq 1}$ of natural numbers with $1 \leq j_k \leq m_{n_k}$ for all $k\geq 1$ such that
\begin{equation}
\tr_{\D_{n_k, j_k}}(a_{n_k, j_k}) + \epsilon < \alpha
\label{Jan220152AM}
\end{equation}
for all $k \geq 1$.  To simplify notation, for all $k \geq 1$ let $\E_k = \D_{n_k, j_k}$ and let $\tau_k$ be the unique tracial state on $\D_{n_k, j_k}$.

Let $\prod_{l=1}^{\infty} \E_l$ denote the $l_{\infty}$-direct sum of
$\{\E_l\}^\infty_{l=1}$ and let $\bigoplus_{l=1}^{\infty} \E_l$ denote the $c_0$-direct
sum.  For all $k \geq 1$, let $\tilde{\tau}_k$ denote the tracial state on
$\prod_{l=1}^{\infty} \E_l$ defined by
\[
\tilde{\tau}_k((y_l)_{l\geq 1}) = \tau_k(y_k).
\]
Since $T(\prod_{l=1}^{\infty} \E_l)$ is weak$^*$ compact,
there exists a subnet $(\tilde{\tau}_{k_{\alpha}} )_{\alpha \in I}$ of
$( \tilde{\tau}_k )_{k\geq 1}$ and a
$\mu \in T(\prod_{l=1}^{\infty} \E_l)$ such that
\[
\mu = \text{w}^*\text{-}\lim_{\alpha} \tilde{\tau}_{k_{\alpha}}.
\]
By definition, it is clear that
\[
\bigoplus_{l=1}^{\infty} \E_l \subseteq \ker(\mu)
\]
and thus $\mu$ induces a tracial state on $\prod_{l=1}^{\infty} \E_l/
\bigoplus_{l=1}^{\infty} \E_l$ which will also be denoted $\mu$.

By previous discussions, there exists a well-defined unital $^*$-homomorphism
\[
\Phi : \A \rightarrow \prod_{l=1}^{\infty} \E_l /   \bigoplus_{l=1}^{\infty} \E_l
\]
defined for all $x \in \A$ by
\[
\Phi(x) = [ (x_{n_l, j_l} )_{l=1}^{\infty}].
\]
Therefore $\mu \circ \Phi$ is a tracial state on $\A$.  However by the definitions of $\Phi$ and $\mu$, and by (\ref{Jan220152AM}),  
\[
 \mu(\Phi(a)) + \epsilon = \lim_{\alpha} \tau_{k_{\alpha}}(a_{n_{k_{\alpha}}, j_{k_{\alpha}}}) + 
\epsilon \leq \alpha  
\]
which contradicts our hypotheses on $a$ and $\alpha$.  Hence (1) must be true.

The proof of (2) will invoke a more complicated version of the asymptotic algebra argument used above so some of the notation will be similar.

To begin, suppose that $a, b$, and $\epsilon$ satisfy the hypotheses of (2).   If necessary, choose $\delta_0 > 0$ so that $a + \delta_0 1_\A$ and $b + \delta_0 1$ are positive.  Since $\lambda^\tau_{a + \delta_0 1_\A}(s) = \delta_0 + \lambda^\tau_{a}(s)$ and $\lambda^\tau_{b + \delta_0 1_\A}(s) = \delta_0 + \lambda^\tau_{b}(s)$ by part (\ref{EFTranslation}) of Theorem \ref{thm:EFunctionProperties}, we may assume without loss of generality that $a,b \in \A_+$.

Since $a, b$ have finite spectrum, let
$\{p_k\}^L_{k=1}, \{q_k\}^M_{k=1} \subseteq \A$ be collections of pairwise orthogonal, non-zero projections
and let $\{\alpha_k\}^L_{k=1}, \{\beta_k\}^M_{k=1} \subseteq \mathbb{R}$ be such that  
\begin{enumerate}[(a)]
\item $\sum^L_{k=1} p_k = \sum^M_{k=1} q_k = 1_\A$,
\item $\alpha_1 > \alpha_2 > \cdots > \alpha_L > 0$, 
\item $\beta_1 > \beta_2 > \cdots > \beta_M > 0$,  
\item $a = \sum_{j=1}^L \alpha_j p_j$, and
\item $b = \sum_{k=1}^M \beta_k q_k$.
\end{enumerate}
Since $\A$ is unital and simple, and since $p_1 \neq 0$ and $q_1 \neq 0$,
\[
\inf_{\tau \in T(\A)} \tau(p_1) > 0 \qqand 
\inf_{\tau \in T(\A)} \tau(q_1) > 0.
\]
Let $\hat{t} > 0$ be such that 
\[
\hat{t} < \min\left\{ \inf_{\tau \in T(\A)} \tau(p_1), 
\inf_{\tau \in T(\A)} \tau(q_1) \right\}.
\]
Hence, for all $0 \leq s \leq \hat{t}$ and $\tau \in T(\A)$, we have that
\[
\lambda_a^{\tau}(s) = \alpha_1
\qqand 
\lambda_b^{\tau}(s) = \beta_1.
\]
Since, by the hypotheses on $a$ and $b$, for all $\tau \in T(\A)$ and $t \in [0, \hat{t}]$ we have that
\[
\int_0^t \lambda_a^{\tau}(s) \, ds < \epsilon t + \int_0^t 
\lambda_b^{\tau}(s) \, ds,
\]
it must be the case that
\[
\alpha_1 t < \epsilon t + \beta_1 t
\]
for all $t  \in [0, \hat{t}]$.  Hence
\[
\alpha_1 < \epsilon +  \beta_1.
\]

Let $f, g :[0, \infty) \rightarrow [0,1]$
be the continuous functions defined by
\[
f(s) = 
\begin{cases}
0 & s \in [0, \frac{\alpha_1 + \alpha_2}{2}]\\
1 & s \in [\alpha_1, \infty)\\
\makebox{  linear on  } & [\frac{\alpha_1 + \alpha_2}{2}, \alpha_1]
\end{cases}
\]
and
\[
g(s) = 
\begin{cases}
0 & s \in [0, \frac{\beta_1 + \beta_2}{2}]\\
1 & s \in [\beta_1, \infty)\\
\makebox{  linear on  } & [\frac{\beta_1 + \beta_2}{2}, \beta_1]
\end{cases}.
\]
Note that $f(a) = p_1$ and $g(b) = q_1$.

Now, suppose to the contrary there exists an $\epsilon' > \epsilon$, sequences $(n_k )_{k\geq 1}$ and $( j_k)_{k\geq 1}$ of natural numbers with $1 \leq j_k \leq m_{n_k}$ for all $k\geq 1$, and a sequence $( t_k)_{k\geq 1}$ with $t_k \in [0,1]$ for all $k\geq 1$ such that
\begin{equation}
\int_0^{t_k} \lambda^{\tr_{\D_{n_k,j_k}}}_{a_{n_k,j_k}}(s) \, ds 
> \epsilon' t_k + \int_0^{t_k} \lambda^{\tr_{\D_{n_k,j_k}}}_{b_{n_k,j_k}}(s) \, ds
\label{equ:Christmas20158:54PM}
\end{equation}
for all $k \geq 1$.

Notice by part (1) of this result that 
\[
\tr_{\D_{n_k,j_k}}(f(a_{n_k,j_k})), \tr_{\D_{n_k,j_k}}(g(b_{n_k,j_k}))> \hat{t}
\]
provided $k$ is sufficiently large.  Furthermore, by the construction preceding this lemma, we may assume that $\sp(a_{n_k,j_k}) \subseteq \sp(a)$ and $\sp(b_{n_k,j_k}) \subseteq \sp(b)$ for all $k \geq 1$.  Hence, for sufficiently large $k$, we obtain that
\[
\lambda_{a_{n_k,j_k}}^{\tr_{\D_{n_k,j_k}}}(s) = \alpha_1 < \epsilon +  
\beta_1 = \epsilon +  \lambda_{b_{n_k,j_k}}^{\tr_{\D_{n_k,j_k}}}(s)
\] 
for all $s \in [0, \hat{t}]$. Thus 
\[
\int_0^t \lambda_{a_{n_k,j_k}}^{\tr_{\D_{n_k,j_k}}}(s) \, ds 
< \epsilon t + \int_0^t \lambda_{b_{n_k,j_k}}^{\tr_{\D_{n_k,j_k}}}(s) \, ds
\] 
for $t \in [0, \hat{t}]$ provided $k$ is sufficiently large.  Therefore, as $\epsilon' > \epsilon$, 
passing to a subsequence of $( n_k )_{k\geq 1}$ if necessary, we may
assume that $t_k > \hat{t}$ for all $k \geq 1$.

Repeating ideas from the first part of this proof, for all $k \geq 1$ let $\E_k = \D_{n_k, j_k}$, let $\tau_k$ be the unique tracial state on $\D_{n_k, j_k}$, and let $\prod_{l=1}^{\infty} \E_l$ and $\bigoplus_{l=1}^{\infty} \E_l$ be as before.  For all $k \geq 1$ let $\tilde{\tau}_k$ denote the tracial state on $\prod_{l=1}^{\infty} \E_l$ defined by
\[
\tilde{\tau}_k((y_l)_{l\geq 1}) = \tau_k(y_k).
\]
Since $T(\prod_{l=1}^{\infty} \E_l)$ is weak$^*$ compact,
there exists a subnet $(\tilde{\tau}_{k_{\alpha}} )_{\alpha \in I}$ of
$( \tilde{\tau}_k )_{k\geq 1}$ and a
$\mu \in T(\prod_{l=1}^{\infty} \E_l)$ such that
\[
\mu = \text{w}^*\text{-}\lim_{\alpha} \tilde{\tau}_{k_{\alpha}}.
\]
Finally, since $t_k > \hat{t}$ for all $k\geq 1$ and since $[\hat{t}, 1]$ 
is compact, passing to a subnet if necessary, we may assume that
there exists $r \in [\hat{t},1]$ such that 
$\lim_\alpha t_{k_{\alpha}} = r$.

By definition, it is clear that
\[
\bigoplus_{l=1}^{\infty} \E_l \subseteq \ker(\mu)
\]
and thus $\mu$ induces a tracial state on $\prod_{l=1}^{\infty} \E_l/
\bigoplus_{l=1}^{\infty} \E_l$ which will also be denoted $\mu$.  
By previous discussions, there exists a well-defined unital $^*$-homomorphism
\[
\Phi : \A \rightarrow \prod_{l=1}^{\infty} \E_l /   \bigoplus_{l=1}^{\infty} \E_l
\]
defined for all $x \in \A$ by
\[
\Phi(x) = [ (x_{n_l, j_l} )_{l=1}^{\infty}].
\]
Therefore $\mu \circ \Phi$ is a tracial state on $\A$. 

We claim for all $d \in \A_+$ and $t \in [0,1]$ that
\[
\int_0^t \lambda_d^{\mu\circ \Phi} (s) \, ds = \lim_{\alpha}
\int_0^t \lambda_{d_{n_{k_{\alpha}}}}^{\tau_{k_{\alpha}}} (s) \, ds.
\]
To see this, first note by Lemma \ref{lem:IntLambdaChar} that
\[
\int_0^t \lambda_d^{\mu \circ \Phi} (s) \, ds = 
\sup\{ \mu (\Phi(d c)) \, \mid \, c \in \A,
 0 \leq c \leq 1_\A, \mu (\Phi(c)) = t \}.
 \]
Let $\delta > 0$ be arbitrary and choose $c \in \A$ with $0 \leq c \leq 1_\A$ such that 
$\mu (\Phi(c)) = t$  
and
\[
\mu(\Phi(dc) )
+ \delta > \int_0^t \lambda_d^{\mu \circ \Phi} (s) \, ds.
\]
Notice
\[
\mu (\Phi(dc) )
= \lim_{\alpha} \tau_{k_{\alpha}}(d_{n_{k_{\alpha}}} c_{n_{k_{\alpha}}})
\qand 
\lim_{\alpha} \tau_{k_{\alpha}}(c_{n_{k_{\alpha}}}) = t
\]
(as $\tau_{k_{\alpha}}(d_{n_{k_{\alpha}}} c_{n_{k_{\alpha}}}) = \tau_{k_{\alpha}}(d_{n_{k_{\alpha}}, j_{k_\alpha}} c_{n_{k_{\alpha}}, j_{k_\alpha}})$). Since Lemma \ref{lem:IntLambdaChar} implies
\begin{align*}
0 \leq \tau_{k_{\alpha}}(d_{n_{k_{\alpha}}} c_{n_{k_{\alpha}}}) & \leq \int^{\tau_{k_{\alpha}}(c_{n_{k_{\alpha}}})}_0  \lambda^{\tau_{k_\alpha}}_{d_{n_{k_{\alpha}}}}(s)\, ds \\
& = \int^t_0 \lambda^{\tau_{k_\alpha}}_{d_{n_{k_{\alpha}}}}(s)\, ds - \int^t_{\tau_{k_{\alpha}}(c_{n_{k_{\alpha}}})} \lambda^{\tau_{k_\alpha}}_{d_{n_{k_{\alpha}}}}(s)\, ds
\end{align*}
and part (\ref{EFNorm}) of Theorem \ref{thm:eigenvalue-function-properties} implies that
\[
\left|\int^t_{\tau_{k_{\alpha}}(c_{n_{k_{\alpha}}})} \lambda^{\tau_{k_\alpha}}_{d_{n_{k_{\alpha}}}}(s)\, ds\right| \leq \left|t - \tau_{k_{\alpha}}(c_{n_{k_{\alpha}}})\right| \left\|d_{n_{k_{\alpha}}}\right\| \leq \left|t - \tau_{k_{\alpha}}(c_{n_{k_{\alpha}}})\right| \left\|d\right\|,
\]
we have for all sufficiently large $\alpha$ that
\[
\int_0^t \lambda_{d_{n_{k_{\alpha}}}}^{\tau_{k_{\alpha}}} (s) \, ds 
+ \frac{1}{2}\delta > \int_0^t \lambda_d^{\mu \circ \Phi} (s) \, ds.
\]
Therefore, since $\delta > 0$ was arbitrary,
we have that 
\[
\liminf_{\alpha} \int_0^t 
\lambda_{d_{n_{k_{\alpha}}}}^{\tau_{k_{\alpha}}} (s) \, ds 
\geq \int_0^t \lambda_d^{\mu \circ \Phi} (s) \, ds.
\]

To complete the claim, notice by Lemma \ref{lem:IntLambdaChar} that for each $k \geq 1$ there exists a $\tilde{c}_k \in \E_{k}$ such that
$0 \leq \tilde{c}_k \leq 1_{\E_k}$, $\tau_{k}(\tilde{c}_k) = t$, and
\[
\tau_k(d_{n_k} \tilde{c}_{k}) + \frac{1}{k}
> \int_0^t 
\lambda_{d_{n_k}}^{\tau_{k}} (s) \, ds.
\] 
Let $\tilde{c} = ( \tilde{c}_{l})_{l\geq 1}$ and $\tilde{d} = ( {d}_{n_l, j_l})_{l\geq 1}$ which are elements of $\prod_{l=1}^{\infty} \E_l$. Therefore 
\begin{align*}
 \mu(\tilde{d} \tilde{c}) &= \lim_{\alpha} \tau_{k_{\alpha}}
(d_{n_{k_{\alpha}}} \tilde{c}_{k_{\alpha}})\\
& \geq  \limsup_{\alpha} \left( -\frac{1}{{k_{\alpha}}} + \int_0^t 
\lambda^{\tau_{k_{\alpha}}}_{d_{n_{k_{\alpha}}}} (s) \, ds \right) \\
& =  \limsup_{\alpha} \int_0^t 
\lambda^{\tau_{k_{\alpha}}}_{d_{n_{k_{\alpha}}}} (s) \, ds.
\end{align*}
Furthermore, we clearly have that $0 \leq \tilde{c} \leq 1$ and 
$\mu(\tilde{c}) = \lim_{\alpha} \tau_{k_{\alpha}}(\tilde{c}_{k_{\alpha}})
= t$.
Therefore, by Lemma \ref{lem:IntLambdaChar}, 
\[
\int_0^t \lambda^{\mu}_{\tilde{d}} (s) \, ds \geq
\limsup_{\alpha} \int_0^t 
\lambda^{\tau_{k_{\alpha}}}_{d_{n_{k_{\alpha}}}} (s) \, ds.
\]   
Thus, as $\lambda^{\mu \circ \Phi}_{d} = \lambda^{\mu}_{\tilde{d}}$, we obtain that
\[
\int_0^t \lambda^{\mu \circ \Phi}_{d}(s) \, ds \geq 
\limsup_{\alpha} \int_0^t 
\lambda^{\tau_{k_{\alpha}}}_{d_{n_{k_{\alpha}}}} (s) \, ds
\] 
thereby completing the proof of the claim.

Choose $\epsilon''$ such that $\epsilon < \epsilon'' < \epsilon'$.  Since $a_{n_k, j_k}$ and $b_{n_k, j_k}$ are self adjoint elements such that 
$\sp(a_{n_k, j_k}) \subseteq \sp(a)$ and $\sp(b_{n_k, j_k}) \subseteq \sp(b)$ for all $k \geq 1$ by construction, we have that $\lambda^{\tau_k}_{a_{n_k, j_k}}$ and
$\lambda^{\tau_k}_{b_{n_k, j_k}}$ are bounded in the $L_\infty$-norm by Theorem \ref{thm:eigenvalue-function-properties}.  Therefore, as $\lim_\alpha t_{k_\alpha} = r$, we obtain by
(\ref{equ:Christmas20158:54PM}) that 
\[
\int_0^{r} \lambda^{\tau_{k_{\alpha}}}_{a_{n_{k_{\alpha}}, 
j_{k_{\alpha}}}} (s) \, ds > \epsilon'' r 
+ \int_0^{r} \lambda^{\tau_{k_{\alpha}}}_{b_{n_{k_{\alpha}}, 
j_{k_{\alpha}}}} (s) \, ds,
\]
for all sufficiently large $\alpha$.  Hence, by taking the limit and applying the claim, we have that 
\[
\int_0^{r} \lambda^{\mu \circ \Phi}_{a}
(s) \, ds \geq \epsilon'' r 
+ \int_0^{r} \lambda^{\mu \circ \Phi}_{b}
(s) \, ds,
\]
which contradicts the assumptions on $a, b$ (as $\epsilon'' > \epsilon$
and $r > 0$).  Hence the first inequality of part (2) has been demonstrated.

The proof of the second inequality in part (2) follows by similar arguments (or by repeating the proof of the first inequality with $t_k = 1$ for all $k$ and reversing the roles of $a$ and $b$).
\end{proof}

\section{Approximations in TAF C$^*$-Algebras}

One may interpret Lemma \ref{lem:AsymptoticArgument} as saying that in a TAF C$^*$-algebra, if one self-adjoint operator is close to majorizing another with respect to every tracial state, then so are the finite dimensional approximates.   Thus the first goal of this section is to demonstrate that if one self-adjoint matrix $A$  is close to majorizing another self-adjoint matrix $B$, then $B$ is almost in the closed convex hull of the unitary orbit of $A$ (see Lemma \ref{lem:NormPerturbation}).  Once the finite dimensional portion is handled, we will demonstrate that the remainder outside the finite dimensional C$^*$-algebra can also be manipulated in a similar manner.

To begin, we note the following useful result pertaining to majorization of functions, which is elementary to prove.
\begin{lem}
\label{lem:majorization-technical-trick-with-negatives}
Let $f$ and $g$ be bounded, non-increasing functions on $[0,1]$.  Suppose that there exists a $t_0 \in [0,1]$ such that
\begin{itemize}
\item $\int^t_0 f(s) \, ds \leq \int^t_0 g(s) \, ds$ for all $t \in [0,t_0]$,
\item $\int^1_0 f(s) \, ds = \int^1_0 g(s) \, ds$, and
\item $f(s) \geq g(s)$ for all $s \in [t_0, 1]$.
\end{itemize}
Then $\int^t_0 f(s) \, ds \leq \int^t_0 g(s) \, ds$ for all $t \in [0,1]$.
\end{lem}

\begin{lem}
\label{lem:NormPerturbation}
Let $\epsilon > 0$ and let $r \in (0,1)$.
There exists a $\delta > 0$ such that
if
\begin{enumerate}
\item $n \geq 1$,
\item $A, B \in (\M_n)_+$,
\item $\tr(P_{\ker(B)}) \geq r$ (where $P_{\ker(B)}$ is the projection onto the kernel of $B$),
\item $\tr(B) \leq \tr(A) + \delta$ (i.e.
$\int_0^1 \lambda^{\tr}_B (s) \, ds \leq \delta +
\int_0^1 \lambda^{\tr}_A(s) \, ds$), \label{assum:trace}
\item $\int_0^t \lambda_A^{\tr} (s) \, ds \leq
\delta t + \int_0^t \lambda_B^{\tr}(s) \, ds$ for all $t \in [0, 1- \tr(P_{\ker(B)})]$, and \label{assum:near-zero}
\item $\int_0^t \lambda_A^{\tr} (s) \, ds \leq \delta t + \epsilon( t - (1- \tr(P_{\ker(B)})))
+ \int_0^t \lambda_B^{\tr}(s) \, ds$ for all $t \in [1- \tr(P_{\ker(B)}), 1]$, \label{assume:not-too-big-near-one}
\end{enumerate}
then there exists a self-adjoint $B' \in \M_n$ such that
$\| B' - B \| < 2 \epsilon$ and $A \prec_\tr B'$.   In fact, one may take
\[
\delta = \min \left\{ \frac{\epsilon}{3}, \frac{\epsilon r}{4}
\right\} \leq \frac{1}{4}
\epsilon r.
\]
\end{lem}

\begin{proof}
Fix $n \geq 1$ and choose positive contractions $A, B \in \M_n$ that satisfy the
assumptions of the lemma with $\delta$ as described.  Without loss of generality, we may assume that
\[
A = \diag(\alpha_1, \alpha_2, \ldots, \alpha_n) \qqand B = \diag(\beta_1, \beta_2, \ldots, \beta_n)
\]
where $w \geq \alpha_1 \geq \alpha_2 \geq \cdots \geq \alpha_n \geq 0$ and $w \geq \beta_1 \geq \beta_2 \geq \cdots \geq \beta_n \geq 0$.

Since $\tr(P_{\ker(B)}) \geq r$, $\ker(B) \neq \{0\}$ so we may select $k \in \{0, 1,\ldots, n-1\}$ such that $\beta_{k+1} = 0$ yet $\beta_k \neq 0$ (or $k = 0$).  Notice
\begin{equation}
\tr(P_{\ker(B)}) = \frac{n-k}{n}  \geq r.
\label{equ:Jan2720165PM}
\end{equation}

Let
\[
B_0 = \diag \left( \beta_1 + \delta, \ldots, \beta_{k} + \delta, 0, 0, \ldots, 0 \right).
\]
Therefore $\left\|B - B_0\right\| < \delta$ and, by assumption (\ref{assum:near-zero}),
\begin{equation}
\int^{t}_0 \lambda^\tr_A(s) \, ds \leq \int^t_0 \lambda^\tr_{B_0}(s) \, ds
\label{equ:June2720166PM}
\end{equation}
for all $t \leq \frac{k}{n}$.

The remainder of the proof is divided into two cases:

\underline{\textit{Case 1}:  $\tr(B_0) \geq \tr(A)$.}  Note that $\tr(B_0) - \frac{k}{n}\delta = \tr(B)$.  Hence, by assumption (\ref{assum:trace}),
\[
\tr(B_0)  - \frac{k}{n}\delta < \tr(A) + \delta \quad \text{ so }\quad \tr(B_0) < \tr(A) + 2 \delta.
\]
Hence
\[
\tr(A) \leq \tr(B_0) \leq \tr(A) + 2 \delta.
\]
Therefore there exists an $0 < \alpha \leq 2 \delta$
such that
\[
\tr(B_0) - \alpha = \tr(A).
\]

By the definition of $\delta$ and by (\ref{equ:Jan2720165PM}),
\[
2 \delta \leq \frac{1}{2}\epsilon r \leq \frac{\epsilon}{2}
\frac{n -k}{n}.
\]
Therefore
\[
0 < \frac{\alpha}{n-k} \leq \frac{\epsilon}{2n}.
\]
Let 
\[
x = \frac{\alpha n}{n-k}.
\]
Thus $0 \leq x \leq \frac{\epsilon}{2}$.  Finally define
\[
B' = \diag(\beta_1 + \delta, \beta_2 + \delta, \ldots, \beta_k + \delta,
-x, -x, \ldots, -x).
\]
By construction $\left\|B' - B\right\| < \epsilon$ and $\tr(B') = \tr(A)$.  

To see that $A \prec_\tr B'$, notice that $\lambda^\tr_{B'}(s) = \lambda^\tr_{B_0}(s)$ for all $s \in \left[0, \frac{k}{n}\right)$ and thus
\[
\int^t_0 \lambda^\tr_{B'}(s) \, ds= \int^t_0 \lambda^\tr_{B_0}(s) \, ds\geq \int^t_0 \lambda^\tr_{A}(s)\, ds
\]
for all $t \in \left[0, \frac{k}{n}\right)$.  Furthermore, notice that if $s \in \left[\frac{k}{n}, 1\right)$, then
\[
\lambda^\tau_{B'}(s) = -x \leq 0 \leq \lambda^\tau_{A}(s) 
\]
since $A$ is positive.  Hence $A \prec_\tr B'$ by Lemma \ref{lem:majorization-technical-trick-with-negatives}.

\underline{\textit{Case 2}:  $\tr(B_0) < \tr(A)$.}  Note by assumption (\ref{assume:not-too-big-near-one}) that
\[
\tr(A) \leq \tr(B_0) + (\delta + \epsilon)\left(\frac{n-k}{n} \right).
\]
Hence there exists an $m_1 \in \{1, \ldots, n-k\}$ and a $\gamma \in [0, \delta + \epsilon]$ such that
\[
\tr(A) = \tr(B_0) + (\delta + \epsilon) \frac{m_1-1}{n} + \gamma \frac{1}{n}.
\]

Let
\[
B' = \diag(\beta_1 + \delta, \beta_2 + \delta,  \ldots, \beta_k + \delta, \underbrace{\delta + \epsilon,   \ldots, \delta + \epsilon}_{m_1-1}, \gamma, 0, \ldots, 0).
\]
By construction $\left\|B' - B\right\| < 2\epsilon$ and $\tr(B') = \tr(A)$.  

To see that $A \prec_\tr B'$, notice that
\[
\lambda^\tr_{B'}(s) \geq \lambda^\tr_{B_0}(s) = \lambda^\tr_{B}(s) + \delta
\]
for all $s \in \left[0, \frac{k}{n}\right)$ and thus
\[
\int^t_0 \lambda^\tr_{B'}(s) \, ds \geq  \int^t_0 \lambda^\tr_{B_0}(s) \, ds = \delta t +  \int^t_0 \lambda^\tr_{B}(s) \, ds\geq \int^t_0 \lambda^\tr_{A}(s)\, ds
\]
for all $t \in \left[0, \frac{k}{n}\right)$.  Furthermore, notice that if $t \in \left[\frac{k}{n}, \frac{k+m_1-1}{n}\right)$, then
\begin{align*}
\int^t_0 \lambda^\tr_{B'}(s) \, ds & \geq \int^{\frac{k}{n}}_0 \lambda^\tr_{B_0}(s)\, ds + \int^t_{\frac{k}{n}} (\epsilon + \delta) \, ds\\
& \geq \int^{\frac{k}{n}}_0 \lambda^\tr_{B}(s)  \, ds + \delta \frac{k}{n}  + (\epsilon + \delta) \left(t - \frac{k}{n}\right)\\
& = \int^{t}_0 \lambda^\tr_{B}(s)  \, ds + \delta t  + \epsilon \left(t - \frac{k}{n}\right)\\
&\geq \int^t_0 \lambda^\tr_{A}(s)\, ds
\end{align*}
as $\frac{k}{n} = 1 - \tr(P_{\ker(B)})$.  Next, since $\lambda^\tr_{B'}(s)$ and $\lambda^\tr_A(s)$ are constant on $\left[\frac{k+m_1-1}{n}, \frac{k+m_1}{n}\right)$, we note since 
\[
\int^\frac{k+m_1-1}{n}_0 \lambda^\tr_{B'}(s) \, ds \geq \int^\frac{k+m_1-1}{n}_0 \lambda^\tr_{A}(s)\, ds
\]
that
\[
\int^t_0 \lambda^\tr_{B'}(s) \, ds \geq \int^t_0 \lambda^\tr_{A}(s)\, ds
\]
for all $t \in \left[\frac{k+m_1-1}{n}, \frac{k+m_1}{n}\right)$ if and only if
\[
\tr(A) = \tr(B') = \int^\frac{k+m_1}{n}_0 \lambda^\tr_{B'}(s) \, ds \geq \int^\frac{k+m_1}{n}_0 \lambda^\tr_{A}(s)\, ds,
\]
which is true as $A \geq 0$. Finally notice that if $s \in \left[\frac{k+m_1}{n}, 1\right)$, then
\[
\lambda^\tau_{B'}(s) = 0 \leq \lambda^\tau_{A}(s) 
\]
since $A$ is positive.  Hence $A \prec_\tr B'$ by Lemma \ref{lem:majorization-technical-trick-with-negatives}.
\end{proof}

To see the necessity of some lower bound on the size of the kernel for the conclusions of Lemma \ref{lem:NormPerturbation} to hold, for each $n \in \bN$ consider
\[
A_n = \diag(1, 1, \ldots, 1, 0) \qqand B_n = I_n.
\]
It is elementary to see that given a $\delta > 0$ there exists an $n$ such that $A_n$ and $B_n$ satisfy the assumptions of Lemma \ref{lem:NormPerturbation}.  However, it is not possible to perturb $B_n$ to $B'_n$ such that $A_n \prec_\tr B'_n$ unless $\left\|B_n - B'_n\right\| \geq 1$.

Due to the necessity of having non-trivial kernel in our matrix approximates, we will need some control over the trace of the kernel.  This is provided by the following result.
\begin{lem}
\label{lem:size-of-kernel}  
Let $\A$ be a unital, simple C$^*$-algebra and let $a \in \A_+$ such that 
$0 \in \sp(a)$.
Given $\epsilon > 0$ there exists a $\delta > 0$ such that if $b \in \A_+$ and $\left\| a - b \right\| < \delta$ then
\[
\mu_{\tau, b}([0, \epsilon)) \geq \frac{1}{2} \mu_{\tau, a}([0, \epsilon))
\]
for all $\tau \in T(\A)$.
\end{lem}

\begin{proof}
Since $0 \in \sp(a)$ and since $\A$ is unital and simple, for all $r > 0$ it must be the case that
\[
\inf_{\tau \in T(\A)} \mu_{\tau, a}([0,r)) > 0
\]
(i.e. GNS representations are faithful).  

Since $T(\A)$ is weak$^*$-compact, there exists a $\delta_1$ with $0 < \delta_1 < \epsilon$ such that
\[
\mu_{\tau, a}([0, \delta_1)) \geq \frac{2}{3}\mu_{\tau, a}([0, \epsilon))
\]
for all $\tau \in T(\A)$.  Let $f : [0, \infty) \rightarrow [0,1]$ be the  continuous function defined by 
\[
f(t) = 
\begin{cases}
1 & t \in [\epsilon, \infty)\\
0 & t \in [0, \delta_1] \\
\makebox{linear on  } & [\delta_1, \epsilon]
\end{cases}.
\]

Again, since $T(\A)$ is weak$^*$-compact, there exists a $\delta_2 > 0$ such that 
\[
\frac{2}{3} \mu_{\tau, a}([0, \epsilon)) - \delta_2 > \frac{1}{2}\mu_{\tau, a}
([0, \epsilon))
\]
for all $\tau \in T(\A)$. Finally, it is well-known that there exists a $\delta > 0$ such that if $b \in \A_+$ and
$\left\| a - b \right\| < \delta$ then 
\[
\left\| f(a) - f(b) \right\| < \delta_2.
\]
Therefore if $b \in \A_+$ is such that $\left\| a - b \right\| < \delta$, then for all $\tau \in T(\A)$ we have that
\begin{align*}
\mu_{\tau, b}([0, \epsilon))
& \geq  1 - \tau(f(b)) \\
& >  1 - \tau(f(a)) - \delta_2 \\
& \geq   \mu_{\tau, a}([0, \delta_1)) - \delta_2\\
& \geq  \frac{2}{3}\mu_{\tau, a}([0, \epsilon)) - \delta_2 \\
& >  \frac{1}{2} \mu_{\tau, a}([0, \epsilon)). \qedhere
\end{align*} 
\end{proof}

As the above will handle the finite dimensional portions of our self-adjoint operators, the following technical result will be the key in handling the remainders.

\begin{lem}
\label{lem:strict-comparison-convex-hull-provided-room}
Let $\epsilon_1 > 0$, let $0 < \epsilon_2 < 1$, and let $\C$ be a unital, simple  C$^*$-algebra with non-trivial tracial simplex,
strict comparison of projections by tracial states, and cancellation of projections.  Suppose that $p, q$ are non-zero projections in $\C$ and $a, b \in p \C_+ p$ are such that  
\begin{enumerate}
\item $p$ and $q$ are orthogonal,
\item $2 \tau(p) < \tau(q) < \epsilon_1$ for all $\tau \in T(\C)$,
\item $\left\| a \right\|, \left\| b \right\| \leq 1$, and
\item $\sp(a)$ and $\sp( b)$ each have a finite number of points.
\end{enumerate}
Then for every projection $r \in \C$ such that $r$ is orthogonal to $p+q$ and
\[
\tau(r) > 6 \left(\frac{\epsilon_1}{\epsilon_2}\right)
\]
for all $\tau \in T(\C)$,  there exists a C$^*$-subalgebra $\D \subseteq \C$ and an element $c \in r \D_{\sa} r$
such that 
\begin{enumerate}[(a)]
\item $1_{\D} = p + q + r$,
\item $p, q, r, a, b, c \in \D$,
\item $\left\| c \right\| < \epsilon_2$, and
\item $a \in \conv(U_{\D}(b + q + c))$.  
\end{enumerate}
\end{lem}
\begin{proof}

Since $2 \tau(p) < \tau(q)$ for all $\tau \in T(\C)$, there exists mutually orthogonal projections $q_1, q_2 \leq q$ such that $p$ is Murray-von Neumann equivalent to $q_1$ and $q_2$.  Choose $v_1, v_2 \in \C$ such that $v_1^*v_1 = q_1$, $v_2^*v_2 = q_2$, and $v_2v_2^* = v_1v_1^* = p$, and let $q_0 = q - q_1 - q_2$.

Write $a = \sum^n_{k=1} \alpha_k p_k$ where $\alpha_k \in \bR$, $\alpha_k \neq \alpha_j$ for all $k \neq j$, and $\{p_k\}^n_{k=1}$ are pairwise orthogonal projections such that $p = \sum^n_{k=1} p_k$.  If $\D_1 = \mathrm{C}^*(\{q_1, a, v_1\})$, then $\D_1$ is a finite dimensional C$^*$-algebra containing a unitary $u_1 \in \D_1$ such that $u_1^*au_1 \in q_1 \D_1 q_1$.  In particular
\[
\D_1 \cong \left( \M_2\right)^{\oplus n}
\]
where
\begin{align*}
a &= (\alpha_1 \oplus 0) \oplus (\alpha_2 \oplus 0)  \oplus \cdots \oplus (\alpha_n \oplus 0), \\
p &= (1 \oplus 0) \oplus (1 \oplus 0) \oplus \cdots \oplus (1 \oplus 0), \\
q_1 &= (0 \oplus 1) \oplus (0 \oplus 1) \oplus \cdots \oplus (0 \oplus 1), \quad \text{ and}\\
u_1^*au_1 &= (0 \oplus \alpha_1) \oplus (0 \oplus \alpha_2)  \oplus \cdots \oplus (0 \oplus \alpha_n).
\end{align*}
Hence $q_1 \D_1 q_1 \cong \bC^{\oplus n}$.

Similarly, write $b = \sum^m_{k=1} \beta_k p'_k$ where $\beta_k \in \bR$, $\beta_k \neq \beta_j$ for all $k \neq j$, and $\{p'_k\}^m_{k=1}$ are pairwise orthogonal projections such that $p = \sum^m_{k=1} p'_k$.  If $\D_2 = \mathrm{C}^*(\{q_2, b, v_2\})$, then $\D_2$ is a finite dimensional C$^*$-algebra.  In particular,
\[
\D_2 \cong \left( \M_2\right)^{\oplus m}
\]
where
\begin{align*}
b &= (\beta_1 \oplus 0) \oplus (\beta_2 \oplus 0)  \oplus \cdots \oplus (\beta_m \oplus 0), \\
p &= (1 \oplus 0) \oplus (1 \oplus 0) \oplus \cdots \oplus (1 \oplus 0), \quad \text{ and} \\
q_2 &= (0 \oplus 1) \oplus (0 \oplus 1) \oplus \cdots \oplus (0 \oplus 1).
\end{align*}

Notice that $\tau(q_0 + q_1 + 1_{\D_2}) < 2 \epsilon_1$ for all $\tau \in T(\C)$ by the assumptions on $p$ and $q$.  Therefore, since $\tau(r) > 6 \left(\frac{\epsilon_1}{\epsilon_2}\right)$ for all $\tau \in T(\A)$, since $\C$ has strict comparison of projections by
tracial states,  and since $r$ is orthogonal to $p + q$, we may construct finite dimensional C$^*$-algebras $\D_3$, $\D_4$, and $\D_5$ such that $q_1 \D_1 q_1 \subseteq \D_3$, $\D_2 \subseteq \D_4$, $\{q_0\} \subseteq \D_5$, $x$ is orthogonal to $y$ for all $x \in \D_k$ and $y \in \D_j$ whenever $k,j \in \{3,4,5\}$ and $k \neq j$, and, if $\ell = \left\lfloor \frac{3}{\epsilon_2} \right\rfloor$, then
\begin{align*}
\D_3 &\cong \left(\M_{\ell+1}\right)^{\oplus n}, \\
\D_4 &\cong \left( \M_{2(\ell+1)}\right)^{\oplus m}, \quad \text{ and} \\
\D_5 & \cong \M_{\ell+1} \oplus \bC
\end{align*}
in such a way that if
\begin{align*}
r_1 &= (0 \oplus 1_\ell) \oplus (0 \oplus 1_\ell)  \oplus \cdots \oplus (0 \oplus 1_\ell)  \in \D_3, \\
r_2 &= ((0 \oplus 0) \oplus 1_{2\ell}) \oplus   \cdots \oplus ((0 \oplus 0) \oplus 1_{2\ell}) \in \D_4, \quad \text{ and} \\
r_3 &= (0 \oplus 1_\ell) \oplus 1 \in \D_5,
\end{align*}
then $r = r_1 + r_2 + r_3$, and in such a way that in $\D_3$
\begin{align*}
u_1^*au_1 &= (\alpha_1 \oplus 0_\ell) \oplus (\alpha_2\oplus 0_\ell)  \oplus \cdots \oplus (\alpha_n\oplus 0_\ell), \quad \text{ and} \\
q_1 &= (1 \oplus 0_\ell) \oplus (1\oplus 0_\ell)  \oplus \cdots \oplus (1\oplus 0_\ell) ,
\end{align*}
in $\D_4$
\begin{align*}
b &= ((\beta_1 \oplus 0) \oplus 0_{2\ell})  \oplus \cdots \oplus((\beta_m \oplus 0) \oplus 0_{2\ell}), \\
p &= ((1 \oplus 0) \oplus 0_{2\ell})  \oplus \cdots \oplus((1 \oplus 0) \oplus 0_{2\ell}), \quad \text{ and} \\
q_2 &= ((0 \oplus 1) \oplus 0_{2\ell}) \oplus   \cdots \oplus ((0 \oplus 1) \oplus 0_{2\ell}),
\end{align*}
and in $\D_5$
\[
q_0 = (1 \oplus 0_\ell) \oplus 0.
\]
Let $\D$ be the C$^*$-subalgebra of $\C$ generated by $\D_1$, $\D_3$, $\D_4$, and $\D_5$.  Hence, by construction, $1_\D = p + q + r$ and $p,q,r,a,b \in \D$.

Notice if $u = u_1 + q_2 + q_0 + r$, then $u \in \D$ is a unitary such that $u^*au = u_1^*au_1 \in \D$.  We claim that there exists a $c \in r \D_{\sa} r$ such that $\left\| c \right\| < \epsilon_2$ and 
\[
u^*au \in \conv(U_{\D}(b + q + c)).
\]
Note this will complete the proof since this clearly implies $a \in \conv(U_{\D}(b + q + c))$.

Since $u^*a u \in \D_3$, $b \in \D_4$, $q_1 \in \D_3$, $q_2 \in \D_4$, $q_0 \in \D_5$, since $q = q_1 + q_2 + q_0$, and since $\D_3$, $\D_4$, and $\D_5$ are pairwise orthogonal, it suffices to show that there exist self-adjoint $c_3 \in \D_3$, $c_4 \in \D_4$, and $c_5 \in \D_5$ with $\left\|c_j\right\| < \epsilon_2$
and $c_j \in r\D r$ for all $j \in \{3,4,5\}$ such that $u^*au \in \conv(U_{\D_3}(q_1 + c_3))$, $0 \in \conv(U_{\D_4}(b + q_2 + c_4))$, and $0 \in \conv(U_{\D_5}(q_0 + c_5))$.

First we will construct $c_3$.  To begin, we may assume without loss of generality that there is only one copy of $\M_{\ell+1}$ in $\D_3$ and that
\[
u^*a u = \alpha \oplus 0_\ell \qqand q_1 = 1 \oplus 0_\ell
\]
where $0 \leq \alpha \leq 1$ as $0 \leq a \leq 1_\C$.  Let
\[
c_3 = 0 \oplus   \left(- \frac{(1-\alpha)}{\ell}\right) 1_\ell.
\]
Clearly $c_3$ is self-adjoint and
\[
\left\|c_3\right\| = \left|\frac{(1-\alpha)}{\ell}\right| \leq \frac{1}{\ell} \leq \frac{\epsilon_2}{2}
\]
as $\ell = \left\lfloor \frac{3}{\epsilon_2} \right\rfloor$.  Notice that if $\tr_{\ell+1}$ is the normalized tracial state on $\M_{\ell+1}$, then
\[
\tr_{\ell+1}(q_1 + c_3) = \frac{1}{\ell+1} - \ell \frac{(1-\alpha)}{\ell(\ell+1)} = \frac{\alpha}{\ell+1} = \tr_{\ell+1}(u^*au).
\]
Hence, by applying Lemma \ref{lem:majorization-technical-trick-with-negatives} with $t_0 = \frac{1}{\ell}$, $f = \lambda^{\tr_{\ell+1}}_a$, and $g = \lambda^{\tr_{\ell+1}}_{q_1 + c_3}$, we see that $u^*a u \prec_{\tr_{\ell+1}} (q_1 + c_3)$.  Hence $u^*au \in \conv(U_{\D_3}(q_1 + c_3))$ by \cites{A1989, B1946}.

Next we will construct $c_4$.  To begin we may assume without loss of generality that there is only one copy of $\M_{2(\ell+1)}$ in $\D_4$ and
\[
b + q_2 = (\beta \oplus 1) \oplus 0_{2\ell}
\]
where $0 \leq \beta \leq 1$ as $0 \leq b \leq 1_\C$.  Let $\tr_{2(\ell+1)}$ denote the normalized tracial state on $\M_{2(\ell+1)}$.  Then
\[
\tr_{2(\ell+1)}(b + q_2) \in \left[\frac{1}{2(\ell+1)}, \frac{2}{2(\ell+1)}\right].
\]
Let
\[
\gamma = \frac{\ell+1}{\ell} \tr_{2(\ell+1)}(b + q_2) \in \left[\frac{1}{2\ell}, \frac{1}{\ell}\right] \subseteq \left[0, \frac{\epsilon_2}{2}\right].
\]
and let
\[
c_4 = (0\oplus 0) \oplus   \left(- \gamma\right) 1_{2\ell}.
\]
Clearly $c_4$ is self-adjoint  and
\[
\left\|c_4\right\| = |\gamma| \leq \frac{\epsilon_2}{2}.
\]
Furthermore
\[
\tr_{2(\ell+1)}(b + q_2 + c_4) = \tr_{2(\ell+1)}(b + q_2) - \gamma \frac{2\ell}{2(\ell+1)} = 0.
\]
Hence $0 \in \conv(U_{\D_4}(b + q_2 + c_4))$ by \cites{A1989, B1946}.

Since we may construct $c_5$ in a similar manner to $c_4$ so that $0 \in \conv(U_{\D_5}(q_0 + c_5))$, the proof is complete.
\end{proof}

In order to invoke both Lemma \ref{lem:NormPerturbation} and Lemma \ref{lem:strict-comparison-convex-hull-provided-room}, we will need to modify our operators to have a bit of kernel.  The following lemma provides the necessary modification.

\begin{lem}
\label{lem:Budweiser}
Let $\A$ be a unital C$^*$-algebra, let $\tau \in T(\A)$, and let $\{p_k\}^n_{k=1} \subseteq  \A$ be pairwise orthogonal non-zero projections such  that $\sum_{j=1}^n p_j = 1_\A$. Let $a = \sum_{j=1}^n \alpha_j p_j$ where $\alpha_1 > \alpha_2 >  \cdots > \alpha_n \geq 0$.

Suppose  $q_1$ is a subprojection of $p_1$ such that $\tau(q_1) < \tau(p_1)$, $\delta > 0$, and $\beta_1 > 0$ are such that 
\[
\alpha_1 + \delta > \beta_1 > \frac{\alpha_1 \tau(p_1)}{\tau(p_1) - 
\tau(q_1)} > \alpha_1.
\]
If $\tilde{a} = \beta_1 (p_1 - q_1) + \sum_{j=2}^n \alpha_j p_j$, then
\[
\int_0^t \lambda_{a}^{\tau}(s) \, ds 
\leq \int_0^t  \lambda_{\tilde{a}}^{\tau}(s) \, ds 
\leq \delta t + \int_0^t \lambda_{a}^{\tau}(s) \, ds.  
\]
for all $t \in [0,1]$.
\end{lem}

\begin{proof}
Let $\gamma = \tau(p_1)$ and $r = \tau(p_1) - \tau(q_1) > 0$.  Note by the assumptions that
\[
\beta_1 r > \alpha_1 \gamma \qqand r \beta_1 < r(\alpha_1 + \delta).
\]

Using Example \ref{exam:finite-spectrum-eigenvalue-functions}, it is elementary to see that
\[
\lambda^\tau_{\tilde{a}}(s) = \begin{cases}
\beta_1 & \text{if }s < r \\
\lambda^\tau_a(s + \gamma - r) & \text{if } r \leq s < 1 - (\gamma - r)\\
0 & \text{if }1 - (\gamma - r) \leq s < 1
\end{cases}.
\]

To see that $\int_0^t \lambda_{a}^{\tau}(s)\,  ds  \leq \int_0^t  \lambda_{\tilde{a}}^{\tau}(s)\,  ds$ for all $t \in [0,1]$, first notice that if $t \in [0,r]$ then
\begin{align*}
\int_0^t  \lambda_{\tilde{a}}^{\tau}(s) \, ds = t \beta_1 \geq t \alpha_1 =  \int_0^t  \lambda_{a}^{\tau}(s) \, ds.
\end{align*}
Next, if $t \in [r, 1-(\gamma - r)]$ then
\begin{align*}
\int_0^t  \lambda_{\tilde{a}}^{\tau}(s)\,  ds &= r \beta_1 + \int^t_r \lambda^\tau_a(s + \gamma - r) \, ds \\
&\geq \alpha_1 \gamma + \int^{t+ \gamma - r}_\gamma \lambda^\tau_a(x) \, dx \\
&= \int^{t+ \gamma - r}_0 \lambda^\tau_a(x) \, dx \\
& \geq \int^{t}_0 \lambda^\tau_a(x) \, dx
\end{align*}
as $\lambda^\tau_a(x) \geq 0$ for all $x \in [0,1]$.  Finally, if $t \in [1-(\gamma - r), 1]$ then
\begin{align*}
\int_0^t  \lambda_{\tilde{a}}^{\tau}(s)\,  ds = \int_0^{1-(\gamma - r)}  \lambda_{\tilde{a}}^{\tau}(s) \, ds &= r \beta_1 + \int^{1-(\gamma - r)}_r \lambda^\tau_a(s + \gamma - r) \, ds \\
&\geq \alpha_1 \gamma + \int^{1}_\gamma \lambda^\tau_a(x) \, dx \\
&= \int^{1}_0 \lambda^\tau_a(x) \, dx \\
& \geq \int^{t}_0 \lambda^\tau_a(x) \, dx
\end{align*}
as $\lambda^\tau_a(x) \geq 0$ for all $x \in [0,1]$.  Hence the first inequality has been demonstrated.

To see that $\int_0^t  \lambda_{\tilde{a}}^{\tau}(s) \, ds \leq \delta t + \int_0^t \lambda_{a}^{\tau}(s) \, ds$ for all $t \in [0,1]$, first notice that if $t \in [0,r]$ then
\begin{align*}
\int_0^t  \lambda_{\tilde{a}}^{\tau}(s) \, ds = t \beta_1 \leq t (\alpha_1 + \delta) = \delta t + \int_0^t  \lambda_{a}^{\tau}(s) \, ds.
\end{align*}
Next, if $t \in [r, \gamma]$ then
\begin{align*}
\int_0^t \lambda_{\tilde{a}}^{\tau}(s) \,ds & =  \beta_1 r + \int_{r}^t  \lambda^{\tau}_{\tilde{a}}(s) \,ds \\
& \leq  (\alpha_1 + \delta) r + \alpha_1 (t - r) \\
& \leq  \alpha_1 t + \delta r \\
& \leq  \int_0^t \lambda^{\tau}_a (s) \, ds + \delta t
\end{align*} 
as $\lambda^{\tau}_{\tilde{a}}(s) < \alpha_1$ on $[r, \gamma]$.   Next, if $t \in [\gamma, 1-(\gamma - r)]$ (if said interval is non-empty) then
\begin{align*}
\int^t_0 \lambda^\tau_{\tilde{a}}(s) - \lambda^\tau_a(s) \, ds &= \beta_1 r + \int^{t+ \gamma - r}_\gamma \lambda^\tau_a(x) \, dx - \int^t_0 \lambda^\tau_a(s) \, ds \\
&= \beta_1 r + \int^{t+ \gamma - r}_t \lambda^\tau_a(x) \, dx - \int^\gamma_0 \lambda^\tau_a(s) \, ds \\
& \leq \beta_1 r + (\gamma - r)\alpha_1 - \int^\gamma_0 \lambda^\tau_a(s) \, ds  \\
& = \beta_1 r + (\gamma - r)\alpha_1 - \gamma \alpha_1  \\
&= \beta_1 r - \alpha_1 r < r\delta < t \delta.
\end{align*}
Finally, if $t \in [\max\{\gamma, 1-(\gamma - r)\}, 1]$ then, as $\lambda^\tau_{\tilde{a}}(s) = 0$ for all $s \in [1-(\gamma - r), 1]$, we have that 
\begin{align*}
\int^t_0 \lambda^\tau_{\tilde{a}}(s) \, ds &= \int^{1-(\gamma - r)}_0 \lambda^\tau_{\tilde{a}}(s) - \lambda^\tau_a(s) \, ds + \int^{1-(\gamma - r)}_0 \lambda^\tau_a(s) \, ds \\
& \leq (1-(\gamma - r)) \delta + \int^{1-(\gamma - r)}_0 \lambda^\tau_a(s) \, ds \\
& \leq t \delta + \int^t_0 \lambda^\tau_a(s) \, ds
\end{align*}
as $\lambda^\tau_a(s) \geq 0$ for all $s \in [0,1]$.  Hence the second inequality (and thus the lemma) has been demonstrated.
\end{proof}

Finally, we are in a position to combine the above ingredients to obtain our final technical result of this section.  This result says, under minor assumptions, that if $b$ almost majorizes $a$ in a TAF C$^*$-algebra, then $a$ is almost in $\conv(\U(b))$.   For technical purposes in the statement and proof of the result, for each $t > 0$ let $h_{t} : [0, \infty) \rightarrow [0,1]$ denote the continuous function defined by
\begin{equation}
h_{t}(x)  =
\begin{cases}
1 & x \in [0, \frac{2 t}{3}] \\
0 & x \in [t, \infty) \\
\makebox{linear on  } & [\frac{2 t}{3}, \epsilon]
\end{cases}.
\label{equ:h_epsilon}
\end{equation}

\begin{lem}
Let $\A$ be a unital, separable, simple, non-elementary, TAF C$^*$-algebra.  For every $\epsilon > 0$ and for every $\gamma \in (0,1)$ there exists a $\delta > 0$ such that if $a, b \in \A_+$ have finite spectrum, $0 \in \sp(b)$, $\left\| a \right\| \leq \left\| b \right\| = 1$,
\[
\int_0^t \lambda_a^{\tau}(s) \, ds \leq \delta t 
+ \int_0^t \lambda_b^{\tau} (s) \, ds
\qqand 
\int_0^1 \lambda_a^{\tau}(s) \, ds + \delta  > \int_0^1 
\lambda_b^{\tau} (s) \, ds
\]
for all $\tau \in T(\A)$ and $t \in [0,1]$, and
\[
\inf_{\tau \in T(\A)} \tau\left(h_{\frac{1}{1000}\epsilon}(b)\right) > \gamma,
\]
then 
\[
\dist(a, \conv(\U(b))) \leq \epsilon.
\]
\label{lem:TAFFiniteSpectrumCase}
\end{lem}

\begin{proof}
Fix $\epsilon > 0$.  Without loss of generality, we may assume that $\epsilon < 1$.  Let $\delta_0 > 0$ be the resulting $\delta$ from Lemma \ref{lem:NormPerturbation} where $\gamma$ is used for $r$ and $\frac{1}{5} \epsilon$ is used for $\epsilon$.  Without loss of generality, we may assume that $\delta_0 < \frac{1}{10}\epsilon$.  Finally, let $\delta > 0$ be such that 
\[
\delta < \frac{1}{100}\delta_0.
\]
We claim that $\delta$ works for the pair $(\epsilon, \gamma)$.  

Fix $a, b \in \A_+$ satisfying the hypotheses for this $\delta$.  For simplicity, let $\delta_1 = \gamma$. 
Thus
\[
\delta_1 < 
\inf_{\tau \in T(\A)} \mu_{\tau, b}\left(\left[0, \frac{1}{1000}\epsilon\right)\right).
\]

By the assumptions on $b$, we may write $\sp(b) = \{\alpha_k\}^L_{k=1}$ where
\[
1 = \alpha_1 > \alpha_2 >  \cdots > \alpha_L = 0.
\]
Let $f : [0, \infty) \rightarrow [0,1]$ be the continuous function such that 
\[
f(s) 
=
\begin{cases}
1 & s \in [\frac{1 + \alpha_2}{2}, \infty) \\
0 & s \in [0, \alpha_2] \\
\makebox{ linear on } & [\alpha_2, \frac{1 + \alpha_2}{2}]
\end{cases}.
\]
Hence $f(b) \in \A$ is a non-zero projection. Therefore, there exists a $\delta_2 > 0$ such that 
\[
\delta_2 <  \inf_{\tau \in T(\A)} \tau(f(b)).
\]

Choose $\delta_3 > 0$ such that $\delta_3 < \delta_2$ and
\begin{equation}
1 + \frac{1}{100}\delta_0
> \frac{\delta_2}{\delta_2 - 
\delta_3} > 1
\label{equ:TheBeatingsWillContinue} 
\end{equation}
and let 
\[
\delta_4 = \left(\frac{\epsilon}{10000} 
\right)\min \{ \delta_1, \delta_2,
\delta_3 \}.
\]

By Remark \ref{rem:TAFLargeDenominators}, Lemma \ref{lem:AsymptoticArgument}, 
and Lemma \ref{lem:size-of-kernel},  
there exists a finite dimensional C$^*$-algebra $\C \subseteq \A$, $a_0, b_0 \in (1_{\A} - 1_{\C}) \A_+ (1_{\A} - 1_{\C})$, and
$a_1, b_1 \in \C_+$ such that:
\begin{enumerate}
\item $\tau(1_{\A} - 1_{\C}) < \delta_4$ for all $\tau \in 
T(\A)$, \label{SmallNontrivialCorner}  
\item $a_0, a_1, b_0, b_1 \geq 0$,
\item $\left\|a_0 \right\|, \left\| b_0 \right\|, \left\| a_1 \right\|, \left\| b_1 \right\| \leq 1$,  \label{assum:contractions}
\item $\sp(a_0 + a_1) \subseteq \sp(a)$ and 
$\sp(b_0 + b_1) \subseteq \sp(b)$ (so $a_0$, $a_1$, $b_0$, $b_1$ have finite spectrum), \label{spectralassumption}
\item every simple summand of $\C$ has rank greater than
$\frac{1}{\delta_4}$, \label{assum:rank}
\item $\left\| a_0 + a_1 - a \right\|, 
\left\| b_0 + b_1 - b \right\| < \frac{1}{100}\epsilon$,  
\item  $\int_0^t \lambda^{\tau}_{a_1}(s) \, ds 
\leq \frac{\delta_0}{100} t + 
\int_0^t \lambda^{\tau}_{b_1}(s) \, ds$ for all $\tau \in \partial_{\mathrm{ext}} T(\C)$ and $t \in [0,1]$,  \label{AlmostMaj1}   
\item $\int_0^1 \lambda^{\tau}_{a_1}(s) \, ds + \frac{\delta_0}{100}
>  \int_0^1 \lambda^{\tau}_{b_1}(s) \, ds$ for all $\tau \in \partial_{\mathrm{ext}} T(\C)$, \label{AlmostMaj2} 
\item
$\tau(f(b_1))
> \delta_2$  for all $\tau \in \partial_{\mathrm{ext}} T(\C)$, \label{Jan620163:40PM} 
\item 
$\mu_{\tau, b_0 + b_1}([0, \frac{1}{100}\epsilon)) 
\geq \frac{1}{2} \mu_{\tau, b}([0, \frac{1}{100}\epsilon)) 
> \frac{1}{2}\delta_1$ for all $\tau \in T(\A)$, and \label{EmptySpace}
\item 
$\tau\left(h_{\frac{1}{100}\epsilon}(b_1)\right) > \gamma = \delta_1$ for all $\tau \in \partial_{\mathrm{ext}} T(\C)$.   \label{assum:kerneltrace}
 \end{enumerate}

Let $j_0$ be the unique number such that 
$\alpha_{j_0} > \frac{1}{100}\epsilon$ and 
$\alpha_{j_0 +1 } \leq \frac{1}{100}\epsilon$.
Let $g : [0, \infty) \rightarrow [0,1]$ be the function
given by 
\[
g(s) = 
\begin{cases} 
0 & s \in [0, \frac{\alpha_{j_0} + \alpha_{j_0 +1}}{2}]\\
s & s \in [\frac{\alpha_{j_0} + \alpha_{j_0 +1}}{2}, \infty)
\end{cases} .
\]
Then 
by property (\ref{spectralassumption}), $g$ is continuous on $\sp(b_0) \cup \sp(b_1)$ and $\left\|g(b_j) -  b_j\right\| \leq \frac{1}{100}\epsilon$ for $j = 0, 1$.

Since $b_0$ and $b_1$ have finite spectrum, the range projections $r_0$ and $r_1$ of $g(b_0)$ and $g(b_1)$ respectively must be elements of $\A$.  By property (\ref{EmptySpace}),
\[
\tau(1 - (r_0 + r_1)) > \frac{1}{2}\delta_1
\]
for all $\tau \in T(\A)$.  
Since $r_0 \leq 1_{\A} - 1_{\C}$ and $r_1 \leq 1_{\C}$, property (\ref{SmallNontrivialCorner}) implies that
\[
\tau(1_{\A} - 1_{\C}) < \delta_4 < \frac{1}{100}\delta_1
\]
for all $\tau \in T(\A)$.
Let $r_3 = 1_{\C} - r_1$.  Therefore, as $r_3$ is orthogonal to $g(b_1) + (1_{\A} - 1_{\C})$, we obtain that
\[
\frac{\delta_1}{2} < \tau(r_3 + (1_\A - 1_\C - r_0)) < \tau(r_3) + \frac{1}{100}\delta_1
\]
for all $\tau \in T(\A)$, we obtain that
\begin{equation}
\tau(r_3) > \frac{49}{100} \delta_1 \geq \frac{4900\delta_4}{\epsilon}
\label{equ:BeautifulSpace}
\end{equation}
for all $\tau \in T(\A)$.

From properties (\ref{AlmostMaj1}), (\ref{AlmostMaj2}), the definition of
$g$, and 
part (\ref{EFOrder}) of Theorem \ref{thm:EFunctionProperties},  
for all $\tau \in \partial_{\mathrm{ext}} T(\C)$ we have that 
\begin{equation}
\int_0^t \lambda^{\tau}_{a_1}(s) \, ds \leq \frac{\delta_0}{100}t +
\int_0^t \lambda^{\tau}_{g(b_1)}(s) \, ds
\label{equ:AlmostMaj3}
\end{equation}
for all $t \in [0,  \tau(r_1)]$, 
\begin{equation}
\int_0^t \lambda^{\tau}_{a_1}(s) \, ds < \frac{\delta_0}{100}t + 
\frac{\epsilon}{100} (t - \tau(r_1)) +  
\int_0^t \lambda^{\tau}_{g(b_1)}(s) \, ds
\label{equ:AlmostMaj3.1}
\end{equation}
for all $t \in [\tau(r_1), 1]$, and 
\begin{equation}
\int_0^1 \lambda^{\tau}_{a_1}(s) \, ds + \frac{\delta_0}{100} > 
\int_0^1 \lambda^{\tau}_{g(b_1)}(s) \, ds. 
\label{equ:AlmostMaj4}  
\end{equation}

Write $\C = \C_1 \oplus \C_2 \oplus \cdots \oplus \C_M$ 
where each $\C_j$ is a full matrix algebra with rank 
$m_j \geq \frac{1}{\delta_4}$ by property (\ref{assum:rank}).   Furthermore, write $b_1 = \sum^M_{j=1} b_{1,j}$ and
$a_1 =  \sum^M_{j=1} a_{1,j}$ where $a_{1,j}, b_{1,j} \in \C_k$ for all $1 \leq j \leq M$ and let $\tau_j$ be the unique tracial state on 
$\C_j$.   
Hence, by (\ref{equ:AlmostMaj3}), 
(\ref{equ:AlmostMaj3.1}) and (\ref{equ:AlmostMaj4}), we have 
for all $1 \leq j \leq M$ that 
\begin{equation}
\int_0^t \lambda_{a_{1,j}}^{\tau_j}(s) \, ds 
\leq \frac{\delta_0}{100}t 
+ \int_0^t \lambda_{g(b_{1,j})}^{\tau_j}(s) \, ds  
\label{equ:AlmostMaj5}
\end{equation}
for all $t \in [0, \tau_j(r_1)]$,
\begin{equation}
\int_0^t \lambda_{a_{1,j}}^{\tau_j}(s) \, ds 
\leq \frac{\delta_0}{100}t 
+ \frac{\epsilon}{100}(t - \tau_j(r_1)) 
+
\int_0^t \lambda_{g(b_{1,j})}^{\tau_j}(s) \, ds  
\label{equ:AlmostMaj5.1}
\end{equation}
for all $t \in [\tau_j(r_1), 1]$, and 
\begin{equation}
\int_0^1 \lambda_{a_{1,j}}^{\tau_j}(s) \, ds + \frac{\delta_0}{100}   
>  \int_0^1 \lambda_{g(b_{1,j})}^{\tau_j}(s) \, ds.   
\label{equ:AlmostMaj6}
\end{equation}
Furthermore $\tau_j(f(b_{1,j})) > \delta_2$ for all $1 \leq j \leq M$ by property (\ref{Jan620163:40PM}), which, along with property (\ref{spectralassumption}), implies that $1 \in \sp(b_{1,j})$ (and hence $\left\| b_{1,j} \right\| = 1$).  By the definition of $g$, we obtain that  $\tau_j(f(g(b_{1,j}))) > \delta_2$, $1 \in \sp(g(b_{1,j}))$, and $\| g(b_{1,j}) \| = 1$.   Therefore, for each $1 \leq j \leq M$ there exist $p_j, c_j \in \C_+$ such that 
\begin{enumerate}[i.]
\item $p_j$ is a non-zero projection,
\item  $\left\| c_j \right\| < 1$, 
\item  $p_j $ is orthogonal to $c_j$, 
\item $\sp(c_j) \subseteq \sp(g(b_{1,j})) \setminus \{ 1 \} \subseteq \sp(b)  \setminus \{ 1 \}$,
\item $\tau_j(p_j) > \delta_2$, and  
\item  $g(b_{1,j}) = p_j + c_j$.
\end{enumerate}

Since $\text{rank}(\C_j) = m_j \geq \frac{1}{\delta_4}$, we see that
$\frac{1}{m_j} \leq \delta_4 \leq \frac{\epsilon \delta_2}{100}$. 
Hence there exists a projection $q_j \in \C_j$ such that $q_j \leq p_j$ and
\begin{equation}
3 \delta_4 < \tau_j(q_j) < 4 \delta_4 < \delta_3. 
\label{equ:AnheuserBusch}
\end{equation} 

Since $\tau_j(p_j) > \delta_2 > \delta_3 >  \tau_j(q_j)$ for all $1 \leq j \leq M$, we have that $\delta_2 \tau_j(q_j) < \delta_3 \tau_j(p_j)$.  
Hence
\[
\frac{\tau_j(p_j)}{\tau_j(p_j) - \tau_j(q_j)} < 
\frac{\delta_2}{\delta_2 - \delta_3}
\]
for all $1 \leq j \leq M$.  Therefore, by (\ref{equ:TheBeatingsWillContinue}), we obtain for all $1 \leq j \leq M$ that
\[
1 + \frac{\delta_0}{100} > 
\frac{\delta_2}{\delta_2 - \delta_3}
> \frac{\tau_j(p_j)}{\tau_j(p_j) - \tau_j(q_j)}
> 1.
\]
Hence there exists a $\beta_1 \in \bR$ be such that
\begin{equation}
1 + \frac{\delta_0}{100} >
\beta_1 >
\frac{\tau_j(p_j)}{\tau_j(p_j) - \tau_j(q_j)}
> 1
\label{equ:MoraleHadBetterImprove}
\end{equation}
for all $1 \leq j \leq M$.

For every $1 \leq j \leq M$, let 
\[
b'_{1,j} = \beta_1 (p_j - q_j) + c_j = g(b_{1,j}) - p_j + \beta_1 (p_j - q_j) \in  (\C_j)_+.
\]
Notice by (\ref{equ:MoraleHadBetterImprove}) 
and Lemma \ref{lem:Budweiser} that
\[
\int_0^t \lambda_{g(b_{1,j})}^{\tau_j}(s) \, ds 
\leq \int_0^t \lambda_{b'_{1,j}}^{\tau_j}(s) \, ds 
\leq \frac{\delta_0}{100}t + \int_0^t \lambda_{g(b_{1,j})}^{\tau_j}(s) \, ds
\]
for all $t \in [0,1]$.  Therefore, combining with (\ref{equ:AlmostMaj5}), 
(\ref{equ:AlmostMaj5.1}) and (\ref{equ:AlmostMaj6}), we obtain
for all $1 \leq j \leq M$ that 
\begin{equation}
\int_0^t \lambda^{\tau_j}_{a_{1,j}} (s) \, ds \leq \frac{\delta_0}{100}t
+ \int_0^t \lambda^{\tau_j}_{b'_{1,j}} (s) \, ds
\label{equ:AlmostMaj7}
\end{equation}
 for all $t \in [0, \tau_j(r_1)]$,
\begin{equation}
\int_0^t \lambda^{\tau_j}_{a_{1,j}} (s) \, ds \leq \frac{\delta_0}{100}t
+ \frac{\epsilon}{100} (t - \tau_j(r_1)) + 
\int_0^t \lambda^{\tau_j}_{b'_{1,j}} (s) \, ds
\label{equ:AlmostMaj7.1}
\end{equation}
for all $t \in [\tau_j(r_1), 1]$, and
\begin{equation}
\int_0^1 \lambda^{\tau_j}_{a_{1,j}} (s) \, ds + \frac{\delta_0}{50}
> \int_0^1 \lambda^{\tau_j}_{b'_{1,j}} (s) \, ds.
\label{equ:AlmostMaj8}
\end{equation}

Let $e \in \C$ be the projection given by  
\[
e = \sum_{j=1}^M q_j.
\]
Since $g(b_1) = e + \sum_{j=1}^M (p_j - q_j + c_j)$,
\begin{align*}
\left\| g(b_1) - \left(e + \sum_{j=1}^M b'_{1,j} \right) \right\| &= \left\| 
\sum^M_{j=1} g(b_{1,j})  - \left(e + \sum_{j=1}^M b'_{1,j}\right) \right\| \\
&= \left\|\sum^M_{j=1} p_j - \sum_{j=1}^M q_j - \sum^M_{j=1} \beta_1(p_j - q_j)\right\|   \\
&= |\beta_1 - 1| < \frac{\delta_0}{100}.
\end{align*}
Hence
\begin{equation}
\left\| b - \left(b_0 + e + \sum_{j=1}^M b'_{1,j} \right) \right\| < \frac{(\delta_0 + \epsilon)}{10}.
\label{equ:KingOfBeers}
\end{equation}

By (\ref{equ:AnheuserBusch}),
\[
3 \delta_4 <  \tau_j(q_j) < 4 \delta_4
\]
for all $1 \leq j \leq M$ which implies that 
\[
3 \delta_4 \tau(1_{\C_j}) < \tau(q_j) < 4 \delta_4 \tau(1_{\C_j})
\]
for all $\tau \in T(\A)$ and $1 \leq j \leq M$ (since $\tau(1_{\C_j}) \neq 0$ as $\A$ is simple so $\frac{1}{\tau(1_{\C_j})} \tau$ is the unique tracial state on $\C_j$). Hence
\begin{equation}
3 \delta_4 \tau(1_{\C}) < \tau(e) < 4 \delta_4 \tau(1_{\C}) = 4 \delta_4
\label{equ:Jan620167:30PM}
\end{equation}
for all $\tau \in T(\A)$.  Furthermore, we see that
\begin{equation}
\label{equ:Jan620167:45PM}
\tau(e) > 3 \delta_4(1 - \delta_4) > 2 \delta_4 > 2 \tau(1_{\A} - 1_{\C})
\end{equation}
for all $\tau \in T(\A)$.

By (\ref{equ:BeautifulSpace}), (\ref{equ:Jan620167:30PM}), and (\ref{equ:Jan620167:45PM}), we may invoke Lemma \ref{lem:strict-comparison-convex-hull-provided-room} (with $a_0$ for $a$, $b_0$ for $b$, $1_\A - 1_\C$ for $p$, $e$ for $q$, $4\delta_4$ for $\epsilon_1$, $r_3$ for $r$, and $\frac{1}{100}\epsilon$ for $\epsilon_2$) to obtain a C$^*$-subalgebra
$\D \subseteq \A$ and a $d \in r_3 \D_{\sa} r_3$  such that 
\begin{enumerate}[(a)]
\item  $1_{\D} = (1_{\A} - 1_{\C}) + e + r_3$,
\item $(1_{\A} - 1_{\C}), e, r_3, a_0, b_0, d \in \D$,
\item $\left\| d \right\| < \frac{1}{100}\epsilon$, and
\item $a_0 \in \conv(\U_{\D}(b_0 + e + d))$. 
\end{enumerate} 

Notice that $b'_{1,j} \in \C$ so $b'_{1,j}$ is orthogonal to $1_\A - 1_\C$.  Furthermore, $b'_{1,j}$ is orthogonal to $e$ by construction and $r_3$ is orthogonal to $g(b_1)$.  Hence $b'_{1,j}$ is orthogonal to $\D$ for all $j$.  Consequently, we obtain that
\[
a_0 + \sum_{j=1}^M b'_{1,j} \in \conv\left(\U_{\A}\left(b_0 + e + d + 
\sum_{j=1}^M b'_{1,j}\right)\right).
\] 

Since $\left\| d \right\| < \frac{1}{100}\epsilon$, (\ref{equ:KingOfBeers}) implies that
\[
\left\|b - \left(b_0 + e + d +  \sum_{j=1}^M b'_{1,j} \right)\right\| \leq  \frac{1}{100}\epsilon+ \frac{\delta_0 + \epsilon}{5}.
\]
Therefore
\[
\dist\left(a_0 + \sum_{j=1}^M b'_{1,j}, \conv(\U(b))\right) \leq  \frac{\delta_0}{5} 
+ \frac{21 \epsilon}{100}.
\]

For each $1 \leq j \leq M$, notice that if $P_{\ker(b'_{1,j})} \in \C_j$ is the projection
onto the kernel of $b'_{1,j}$ in $\C_j$, then property  (\ref{assum:kerneltrace}) and the definitions of
$g$ and $b'_{1,j}$ imply that
\[
\tau_j(P_{\ker(b'_{1,j})}) > \gamma = \delta_1.
\]
Notice 
\[
1 - \tau_j(P_{\ker(b'_{1,j})}) \leq 1 - \tau_j(P_{\ker(g(b_{1,j}))}) = \tau_j(r_1).
\]
Therefore equations (\ref{equ:AlmostMaj7}), (\ref{equ:AlmostMaj7.1}),
(\ref{equ:AlmostMaj8}) hold with $1 - \tau_j(P_{\ker(b'_{1,j})})$ replacing $\tau_j(r_1)$.  Hence, by the definition  
of $\delta_0$, imply we may invoke Lemma \ref{lem:NormPerturbation} (with $\gamma$ for $r$, $\frac{1}{5} \epsilon$ for $\epsilon$, $a_{1,j}$ for $A$, and $b'_{1,j}$ for $B$ which produces $\delta_0$) to obtain that 
\[
\dist\left(a_1, \conv\left(U_{\C}\left(\sum_{j=1}^M b'_{1,j}\right)\right)  \right)\leq \frac{\epsilon}{5}.
\]
Hence, as $a_0$ is orthogonal to $\C$, we obtain that
\[
\dist\left(a_0 + a_1, \conv(\U(b)) \right) \leq \frac{\delta_0}{5} + \frac{41 \epsilon}{100}.
\]
Therefore 
\[
\dist(a, \conv(\U(b))) \leq \epsilon. \qedhere
\]
\end{proof}

\section{The Main Result}

In this section, we will complete the proof of our main result (Theorem \ref{thm:main}) via Lemma \ref{lem:TAFFiniteSpectrumCase}.  All that remains in the proof is to show that the hypotheses of Lemma \ref{lem:TAFFiniteSpectrumCase} can be obtained, and to invoke some classification results.  We begin with the following.
\begin{lem}
Let $\C$ be a unital, separable, simple, non-elementary
C$^*$-algebra with $T(\C) \neq \emptyset$. If $a, b \in \C_{\sa}$ are such that $a \prec_{\tau} b$ for a $\tau \in T(\C)$, then
\[
\sp(a) \subseteq \conv(\sp(b)).
\]
\label{lem:MajorizationSpectrum}
\end{lem}

\begin{proof}
Let $\tau \in T(\A)$ be arbitrary.  Since $a \prec_\tau b$, the defining inequalities in Definition \ref{defn:majorization} imply that
\[
\lim_{s \nearrow 1} \lambda^\tau_a(s) \geq \lim_{s \nearrow 1} \lambda^\tau_b(s) \qqand \lim_{s \searrow 0} \lambda^\tau_a(s) \leq \lim_{s \searrow 0} \lambda^\tau_b(s).
\]  
Furthermore, since $\A$ is simple, $\tau$ must be faithful.  Hence the result follows from  part (\ref{EFspectrallimits}) of Theorem \ref{thm:EFunctionProperties}.
\end{proof}

\begin{lem}
Let $\A$ be a unital, separable, simple, non-elementary, TAF C$^*$-algebra.  If $a, b \in \A_{\sa}$ are such that $a \prec_{\tau} b$ for all $\tau \in T(\A)$, then $a \in \cconv(\U(b))$.
\label{lem:TAFAll}
\end{lem}

\begin{proof}
By parts (\ref{EFDilation}) and (\ref{EFTranslation}) of Theorem \ref{thm:EFunctionProperties} and by Lemma \ref{lem:MajorizationSpectrum}, we may replace $a, b$ with $\gamma_1 (a + \gamma_2 1_\A)$, $\gamma_1 (b + \gamma_2 1_\A)$ for appropriate $\gamma_1, \gamma_2 \in \mathbb{R}$ with $\gamma_1 > 0$ if necessary to assume that $a, b \geq 0$, $0 \in \sp(b)$, and $\left\| a \right\| \leq \left\| b \right\| = 1$.

Let $\epsilon > 0$ be arbitrary and let $h_{\epsilon}$ be the function in (\ref{equ:h_epsilon}). Notice $h_{\epsilon}(b) \neq 0$ as $0 \in \sp(b)$.
Therefore, as $\A$ is simple, there exists a $\gamma \in (0, 1)$ such that 
$$\gamma < \inf_{\tau \in T(\A)} \tau(h_{\frac{1}{1000}\epsilon}(b)).$$ 
Furthermore, choose $\alpha > 0$ so that 
$$\gamma < \inf_{\tau \in T(\A)} \tau(h_{\epsilon}(b)) - \alpha.$$
Let $\delta > 0$ be the $\delta$ produced in the conclusions of Lemma \ref{lem:TAFFiniteSpectrumCase} using $\frac{1}{10}\epsilon$ for $\epsilon$ and  $\gamma$ for $\gamma$.   

Since $\A$ has real rank zero, there exist $a', b' \in \A_+$ such that 
\begin{itemize}
\item $a', b'$ have finite spectrum,
\item $\left\| a' - a \right\| , \left\| b' - b \right\| < \min \{ \frac{1}{10}\epsilon, \frac{1}{10}\delta \}$,
\item $0 \in \sp(b')$,
\item $\| h_{\frac{1}{1000}\epsilon}(b') - h_{\frac{1}{1000}\epsilon}(b) \| < \frac{1}{3}\alpha$, and
\item $\left\| a' \right\| \leq \left\| b' \right\| = 1$.
\end{itemize}
Therefore, by part (\ref{EFNorm}) of Theorem \ref{thm:EFunctionProperties}, the hypothesis $a \prec_{\tau} b$ for all $\tau \in T(\A)$ implies that
\[
\int_0^t \lambda_{a'}^{\tau}(s) \, ds \leq \frac{\delta}{5}t +
\int_0^t \lambda_{b'}^{\tau}(s) \, ds
\qand
\int_0^1 \lambda_{a'}^{\tau}(s) \, ds  + \frac{\delta}{5} > 
\int_0^1 \lambda_{b'}^{\tau}(s) \, ds
\]
for all $\tau \in T(\A)$ and $t \in [0,1]$.  Furthermore
$$\inf_{\tau \in T(\A)} \tau(h_{\frac{1}{1000}\epsilon}(b')) > \gamma.$$
Hence, as $\delta$ was chosen via Lemma \ref{lem:TAFFiniteSpectrumCase}, we obtain that 
\[
\dist(a', \conv(\U(b'))) \leq \frac{1}{10}\epsilon
\]
and thus
\[
\dist(a, \conv(\U(b))) \leq \epsilon.
\]
Therefore, as $\epsilon$ was arbitrary, the result follows
\end{proof}

\begin{thm}
\label{thm:main}
Let $\A$ be a unital, separable, simple, non-elementary, C$^*$-algebra with non-empty tracial simplex, real rank zero, 
strict comparison of projections, and cancellation of projections.  If $a, b \in \A_{\sa}$, then $a \prec_{\tau} b$ for all $\tau \in T(\A)$ if and only if $a \in \cconv(\U(b))$.
\end{thm}

\begin{proof}
If $a \in \cconv(\U(b))$ then $a \prec_{\tau} b$ for all $\tau \in T(\A)$ by \cite{S2015}*{Lemma 2.20} (the assumption that $\tau$ must be faithful is not required by the same argument as in Theorem \ref{thm:EFunctionProperties}).

Suppose that $a \prec_{\tau} b$ for all $\tau \in T(\A)$.  By \cite{H2001}*{Theorem 4.5},
there exists a unital C$^*$-subalgebra
$\C \subseteq \A$ such that 
$\C$ is a simple, non-elementary, AH algebra with bounded 
dimension growth and real rank zero,
and if $i : \C \rightarrow \A$ is the inclusion map, then 
$i$ induces an isomorphism  of the K-theory invariants
\[
(K_0(\C), K_0(\C)_+, [1_{\C}], K_1(\C)) \cong
(K_0(\A), K_0(\A)_+, [1_{\A}], K_1(\A)).
\] 
(Note in Lin's theorem and argument, the hypothesis of nuclearity is not necessary.)

Hence  \cite{LinRange}*{Lemma 5.1} and \cite{HuLin}*{Theorem 5.6} imply that every positive element of $\A$ is approximately unitarily equivalent
to a positive element of $\C$. Hence, every self-adjoint element of
$\A$ is approximately unitarily equivalent to a self-adjoint element
of $\C$.

The result then follows from Lemma \ref{lem:TAFAll} and from the fact that
$\C$ is TAF (see \cite{H2000}*{Proposition 2.6} for example). 
\end{proof}

\section{A Bound for the Number of Unitaries Conjugates}

To complete this paper, we desire a bound for the number of unitary conjugates of a self-adjoint operator needed in a convex combination to approximate another self-adjoint operator for the C$^*$-algebras studied in Theorem \ref{thm:main}.  As all of the convex combinations constructed are via matrix algebras, it suffices to bound the number of unitary conjugates in a convex combination for matrix algebras independently of the size of the matrix algebras.   Consequently, the following result will enable a result in the context of Theorem \ref{thm:main}.

\begin{thm}
\label{thm:bounding-number-in-approx-convex-combination}
Let $n \in \bN$.  Then there exists an $f(n) \in \bN$ such that for all $m \in \bN$ and positive contractions $A,B \in \M_m$ with $A \prec_{\tr} B$, there exist unitaries $U_1, \ldots, U_{f(n)} \in \M_m$ and $t_1, \ldots t_{f(n)} \in [0,1]$ such that $\sum^{f(n)}_{k=1} t_k = 1$ and
\[
\left\| A - \sum^{f(n)}_{k=1} t_k U^*_k B U_k\right\| < \frac{4}{n-1} + \frac{5}{n}.
\]
In particular
\[
f(n) = 2^{\frac{1}{6} (n^3 - n)}.
\]
\end{thm}
The proof of Theorem \ref{thm:bounding-number-in-approx-convex-combination} will be developed through a finite number of perturbations in order to make certain assumptions on the spectrum and unnormalized traces of $A$ and $B$, followed by a recursive argument.  Before the first perturbation, we recall a notion of majorization for real $n$-tuples that will be of use notationally.
\begin{defn}
Let $\alpha = (\alpha_1, \ldots, \alpha_m), \beta = (\beta_1,\ldots, \beta_m) \in \bR^m$ and let $(\alpha'_1, \ldots, \alpha'_m)$ and $(\beta'_1, \ldots, \beta'_m)$ be the non-increasing rearrangements of $\alpha$ and $\beta$ respectively; that is, $(\alpha'_1, \ldots, \alpha'_m)$ is obtained by rearranging the terms of $(\alpha_1, \ldots, \alpha_m)$ so that
\[
\alpha'_1 \geq \alpha'_2 \geq \cdots \geq \alpha'_m.
\]
It is said that  $\beta$ majorizes $\alpha$, denoted $\alpha \prec \beta$, if
\begin{enumerate}
\item $\sum^l_{k=1} \alpha'_k \leq \sum^l_{k=1} \beta'_k$ for all $l \in \{1, \ldots, m\}$, and
\item $\sum^m_{k=1} \alpha'_k \leq \sum^m_{k=1} \beta'_k$.
\end{enumerate}
\end{defn}
It is elementary to see that if $A, B \in \M_n$ are self-adjoint, then $A \prec_\tr B$ if and only if the eigenvalue list of $B$ majorizes the eigenvalue list of $A$.

\begin{lem}
\label{lem:fractional-trace}
Let $A,B \in \M_m$ be positive contractions with $A \prec_{\tr} B$.  Then 
for all $n \geq 2$ there exist positive contractions  $A',B' \in \M_m$ such that
 
\begin{enumerate}
\item $A' \prec_{\tr} B'$,
\item $\Tr(A') = \Tr(B') \in \left\{\left. \frac{k}{n-1} \, \right| \, k \in \bN \cup \{0\}\right\}$,
\item $\left\|A - A'\right\| < \frac{2}{n-1}$, and 
\item $\left\|B - B'\right\| < \frac{2}{n-1}$.
\end{enumerate}
\end{lem}
\begin{proof}
First notice that $A_0 = \frac{n-2}{n-1}A$ and $B_0 = \frac{n-2}{n-1} B$ are positive contractions with $0 \leq A_0, B_0 \leq \frac{n-2}{n-1} I_m$ and $A_0 \prec_\tr B_0$.  Furthermore,
\[
\left\|A - A_0\right\| \leq \frac{1}{n-1} \qqand \left\|B - B_0\right\| \leq \frac{1}{n-1}.
\]

Since $A_0 \prec_\tr B_0$, we know $\Tr(A_0) = \Tr(B_0)$.  Let
$t \in [0, \infty)$ such that $m t = \Tr(A_0)$.  Then there exists a $k \in \bN \cup \{0\}$ such that $\frac{k}{n-1} \geq t$ and $\left|\frac{k}{n-1} - t\right| < \frac{1}{n-1}$.  Let $r = \frac{k}{n-1} - t$.  It is then elementary to verify that 
\[
A' = A_0 + r I_m \qqand B' = B_0 + r I_m
\]
are positive contractions that satisfy the requirements.
\end{proof}

The following result will enable us to control the spectrum of $A$.

\begin{lem}
\label{lem:fractional-lower-spectra}
Let $A, B \in \M_m$ be positive contractions such that $A \prec_\tr B$ and
\[
\Tr(A) = \Tr(B) \in  \left\{\left.\frac{k}{n-1} \, \right| \, k \in \bN \cup \{0\}\right\}.
\]
Then there exist positive contractions $A' , B' \in \M_m$ such that 
\begin{enumerate}
\item $A' \prec_\tr B'$, 
\item $\Tr(A') = \Tr(B') \in  \left\{\left.\frac{k}{n} \, \right| \, k \in \bN \cup \{0\}\right\}$,
\item $\left\|A - A'\right\| \leq \frac{2}{n}$, $\left\|B - B'\right\| \leq \frac{2}{n}$, and
\item $\sp(A') \subseteq \left\{\left.\frac{k}{n} \, \right| \, k \in\{0, 1, \ldots, n\}\right\}.$
\end{enumerate}
\end{lem}
\begin{proof}
Let
\[
A_0 = \frac{n-1}{n} A \qqand B_0 = \frac{n-1}{n} B.
\]
Note $A_0 \prec_\tr B_0$, $0 \leq A_0, B_0 \leq \frac{n-1}{n} I_m$, $\left\|A - A_0\right\| \leq \frac{1}{n}$, $\left\|B - B_0\right\| \leq \frac{1}{n}$, and
\[
\Tr(A_0) = \Tr(B_0) \in  \left\{\left.\frac{k}{n} \, \right| \, k \in \bN \cup \{0\}\right\}.
\]

Without loss of generality, we may assume that
\[
A_0 = \diag(\alpha_1, \alpha_2,\ldots, \alpha_m) \qqand B_0 = \diag(\beta_1, \beta_2, \ldots, \beta_m)
\]
where $\alpha_1 \geq \alpha_2 \geq \cdots \geq \alpha_m$ and $\beta_1 \geq \beta_2 \geq \cdots \geq \beta_m$.  For each $k \in \{1,\ldots, m\}$, let
\[
\alpha'_k = \min\left\{ \frac{r}{n} - \alpha_k \, \left| \, r \in \{0, 1,\ldots, n-1\}, \frac{r}{n} \geq \alpha_k \right.\right\}.
\]
Hence $0 \leq \alpha'_k \leq \frac{1}{n}$ for all $k$.

Let
\begin{align*}
A' &= \diag(\alpha_1 + \alpha'_1, \alpha_2 + \alpha'_2, \ldots, \alpha_m + \alpha'_m) \quad \text{and} \\
B' &= \diag(\beta_1 + \alpha'_1, \beta_2 + \alpha'_2, \ldots, \beta_m + \alpha'_m).
\end{align*}
Hence $\left\|A - A'\right\| \leq \frac{2}{n}$ and $\left\|B - B'\right\| \leq \frac{2}{n}$.  Furthermore, since  $0 \leq A_0, B_0 \leq \frac{n-1}{n} I_m$, we see that $A'$ and $B'$ are positive contractions.  Since $\Tr(A) = \Tr(B)$, we clearly have $\Tr(A') = \Tr(B')$.  Clearly
\[
\sp(A') \subseteq \left\{\left.\frac{k}{n} \, \right| \, k \in\{0, 1, \ldots, n\}\right\}
\]
by construction and thus 
\[
\Tr(B') = \Tr(A')  \in  \left\{\left.\frac{k}{n} \, \right| \, k \in \bN \cup \{0\}\right\}.
\]
Finally, clearly $\alpha_k + \alpha'_k \geq \alpha_{k+1} + \alpha'_{k+1}$ for all $k$ and thus the simple inequalities
\[
\sum^l_{k=1} \alpha_k + \alpha'_k \leq \sum^l_{k=1} \beta_k + \alpha'_k
\]
show that $A' \prec_\tr B'$.
\end{proof}

To control the spectrum of $B$, we will first need the following.

\begin{lem}
\label{lem:perturb-up-and-down-for-majorization}
Let $B = \diag(\beta_1, \ldots, \beta_m) \in \M_m$ be a self-adjoint diagonal.  
Let $j_1, j_2$ be two indices such that $j_1 \neq j_2$ and $\beta_{j_1} \geq \beta_{j_2}$.  Furthermore, let $r > 0$, $\beta'_j = \beta_j$ if $j \neq j_1, j_2$, $\beta'_{j_1} = \beta_{j_1} + r$, and $\beta'_{j_2} = \beta_{j_2} - r$.  Finally let
\[
B' = \diag(\beta'_1, \ldots, \beta'_m).
\]
Then $B \prec_\tr B'$.
\end{lem}
\begin{proof}
The result follows from the definition of majorization via checking simple inequalities hold.
\end{proof}

\begin{lem}
\label{lem:upper-fractional-spectra}
Let $B \in \M_m$ be a positive contraction such that
\[
\Tr(B) \in  \left\{\left.\frac{k}{n} \, \right| \, k \in \bN \cup \{0\}\right\}.
\]
Then there exists a positive contraction $B' \in \M_m$ such that 
\begin{enumerate}
\item $B \prec_\tr B'$, 
\item $\left\|B - B'\right\| \leq \frac{1}{n}$, and
\item $\sp(B') \subseteq \left\{\left.\frac{k}{n} \, \right| \, k \in\{0, 1, \ldots, n\}\right\}.$
\end{enumerate}
\end{lem}
\begin{proof}
Let $\beta_1 \geq \beta_2 \geq \cdots \geq \beta_m$ be the eigenvalues of $B$. Without loss of generality, we may assume that
\[
B = \diag(\beta_1, \ldots, \beta_m).
\]
We will recursively increase the number of $\beta_k$ that are in the set $\left\{\left.\frac{k}{n} \, \right| \, k \in \bN \cup \{0\}\right\}$ while ensuring conditions (1) and (2).

At each stage of the recursive process (assuming we are not done), let
\[
k_1 = \min\left\{k \, \left| \, \beta_k \neq \frac{r}{n} \text{ for some }r \in \{0, 1, \ldots, n\} \right.\right\}
\]
and
\[
k_2 = \max\left\{k \, \left| \, \beta_k \neq \frac{r}{n} \text{ for some }r \in \{0, 1, \ldots, n\} \right.\right\}.
\]
Note it must be the case that $k_1 < k_2$ as if $k_1 = k_2$, $\beta_{k_1}$ is the only eigenvalue that is not an integer multiple of $\frac{1}{n}$ which contradicts the fact that $\Tr(B) \in  \left\{\left.\frac{k}{n} \, \right| \, k \in \bN \cup \{0\}\right\}$.

Let
\[
r_1 = \min\left\{ \frac{r}{n} - \beta_{k_1}   \, \left| \, \frac{r}{n} > \beta_{k_1}, r \in \{0, 1, \ldots, n\}  \right. \right\}
\]
and
\[
r_2 = \min\left\{ \beta_{k_2} - \frac{r}{n}    \, \left| \, \frac{r}{n} < \beta_{k_2}, r \in \{0, 1, \ldots, n\}  \right. \right\}.
\]
Furthermore, let $r = \min\{r_1, r_2\}$ and let $B_0 = \diag(\beta'_1, \ldots, \beta'_m)$ where $\beta'_j = \beta_j$ if $j \neq k_1, k_2$, $\beta'_{k_1} = \beta_{k_1} + r$, and $\beta'_{k_2} = \beta_{k_2} - r$.  Therefore $B_0$ is a positive contraction and, by Lemma \ref{lem:perturb-up-and-down-for-majorization}, 
\[
B \prec_\tr B_0.
\]
Notice by construction that either $\beta'_{k_1}$ or $\beta'_{k_2}$ is in $\left\{\left.\frac{k}{n} \, \right| \, k \in\{0, 1, \ldots, n\}\right\}$.  Consequently, by recursively applying this process (where the $\beta'_k$ may now need to be reordered), we must eventually obtain a positive contraction $B'$ such that $B \prec_\tr B'$ and 
\[
\sp(B') \subseteq \left\{\left.\frac{k}{n} \, \right| \, k \in\{0, 1, \ldots, n\}\right\}.
\]
Furthermore, by analyzing the above process, we see that $B'$ is obtained from $B$ by changing the diagonal entries in such a way that if $r \in \{0,\ldots, n-1\}$ is such that $\frac{r}{n} \leq \beta_k \leq \frac{r+1}{n}$, then the $k^\th$ diagonal entry entry of $B'$ is either $\frac{r}{n}$ or $\frac{r+1}{n}$.  Hence $\left\|B - B'\right\| \leq \frac{1}{n}$ as desired.
\end{proof}

To obtain a bound for the number of unitaries needed in a convex combination, we note the following elementary result.

\begin{lem}
\label{lem:a-convex-combination-of-2}
Let $r_1, r_2 \in [0,1]$ be such that $r_1 \geq r_2$.  Then for each $s_1 \in [r_2, r_1]$ there exist $U_1, U_2 \in \M_2$ such that
\[
\diag(s_1, r_1 + r_2 - s_1) = \frac{1}{2}U_1^* \diag(r_1, r_2) U_1 +  \frac{1}{2} U_2^* \diag(r_1, r_2) U_2.
\]
\end{lem}
\begin{proof}
Since $s_1 \in [r_2, r_1]$, there exists a $t \in [0,1]$ such that $s_1 = t r_1 + (1-t) r_2$.  Let $\theta \in [0,2\pi]$ be such that $\cos^2(\theta) = t$.  We claim that the unitaries
\[
U_1 = \begin{bmatrix}
\cos(\theta) & \sin(\theta) \\ - \sin(\theta) & \cos(\theta)
\end{bmatrix}
\qqand
U_2 = \begin{bmatrix}
\cos(\theta) & -\sin(\theta) \\  \sin(\theta) & \cos(\theta)
\end{bmatrix}
\]
work.  Indeed, as $\sin^2(\theta) = 1-t$,
\[
U_1^* \diag(r_1, r_2) U_1 = \begin{bmatrix}
tr_1 + (1-t) r_2 & \cos(\theta) \sin(\theta) (r_1 - r_2) \\ \cos(\theta) \sin(\theta) (r_1 - r_2) & (1-t) r_1 + t r_2
\end{bmatrix}
\]
and
\[
U_2^* \diag(r_1, r_2) U_2 = \begin{bmatrix}
tr_1 + (1-t) r_2 & \cos(\theta) \sin(\theta) (r_2 - r_1) \\ \cos(\theta) \sin(\theta) (r_2 - r_1) & (1-t) r_1 + t r_2
\end{bmatrix}.
\]
Hence
\[
\frac{1}{2}U_1^* \diag(r_1, r_2) U_1 +  \frac{1}{2} U_2^* \diag(r_1, r_2) U_2 = \diag(s_1, (1-t) r_1 + t r_2).
\]
However
\[
((1-t) r_1 + t r_2) + s_1 = r_1 + r_2
\]
so the claim follows.
\end{proof}

To progressively apply Lemma \ref{lem:a-convex-combination-of-2}, we note the following triviality.

\begin{lem}
\label{lem:progressive-convex-combinations}
Let $A, B, C \in \M_m$ be such that $A$ is a convex combination of $k_1$ unitary conjugates of $B$ and $B$ is a convex combination of $k_2$ unitary conjugates of $C$.  Then $A$ is a convex combination of $k_1k_2$ unitary conjugates of $C$.
\end{lem}

\begin{proof}[Proof of Theorem \ref{thm:bounding-number-in-approx-convex-combination}]
By applying Lemma \ref{lem:fractional-trace}, then Lemma \ref{lem:fractional-lower-spectra} to the results, and then by applying Lemma \ref{lem:upper-fractional-spectra} to those results, there exist positive contractions
$A_0, B_0 \in \M_m$ such that
\begin{enumerate}
\item $A_0 \prec_\tr B_0$,
\item $\left\|A_0 - A\right\| < \frac{2}{n-1} + \frac{2}{n}$, $\left\|B_0 - B\right\| < \frac{2}{n-1}+ \frac{3}{n}$, and
\item $\sp(A_0), \sp(B_0) \subseteq \left\{\left.\frac{k}{n} \, \right| \, k \in\{0, 1, \ldots, n\}\right\}.$
\end{enumerate}
If there exist $U_1, \ldots, U_{f(n)} \in \M_m$ and $t_1, \ldots t_{f(n)} \in [0,1]$ such that $\sum^{f(n)}_{k=1} t_k = 1$ and
\[
A_0 = \sum^{f(n)}_{k=1} t_k U^*_k B_0 U_k,
\]
then the result will follow as $\left\|A_0 - A\right\| < \frac{2}{n-1} +
\frac{2}{n}$ and
\[
\left\|\sum^{f(n)}_{k=1} t_k U^*_k B_0 U_k - \sum^{f(n)}_{k=1} t_k U^*_k B U_k\right\| \leq \sum^{f(n)}_{k=1} t_k \left\|B_0 - B\right\| < \frac{2}{n-1} +
\frac{3}{n}.
\]

We will now demonstrate a progressive method using Lemmata \ref{lem:a-convex-combination-of-2} and \ref{lem:progressive-convex-combinations} to obtain that 
\[
A_0 = \sum^{f(n)}_{k=1} t_k U^*_k B_0 U_k.
\]
Note for that which follows, we may assume that $A_0$ and $B_0$ are both diagonal matrices with the eigenvalues along the diagonal in non-increasing order.

For a positive contraction $C \in \M_m$, let $\lambda(C; x)$ be the number of eigenvalues of $C$ (counting multiplicity) equal to $x$.  By assumption, we know that
\[
\sum^n_{k=0} \lambda\left(A_0; \frac{k}{n}\right) =\sum^n_{k=0} \lambda\left(B_0; \frac{k}{n}\right) = m.
\]
Since $A_0 \prec_\tr B_0$, it must be the case that $\lambda(A_0; 1) \leq \lambda(B_0; 1)$.  By using direct sums of unitaries with identity matrices and by decreasing the size of $m$, we may assume that $\lambda(A_0; 1) = 0$ (i.e. both $A_0$ and $B_0$ are diagonal matrices with $\lambda(A_0; 1)$ and $\lambda(B_0; 1)$ ones respectively along the diagonal; we may ignore the first $\lambda(A_0; 1)$ diagonal entries).

If $\lambda(B_0; 1) = 0$, we may move on to dealing with $\frac{n-1}{n}$ instead of $1$.  Otherwise, the first $\lambda(B_0; 1) > 0$ diagonal entries of $B_0$ need to be `corrected'.  By `corrected' we mean we want to replace $B_0$ with $B_1$ where 
\begin{enumerate}
\item $A_0 \prec_\tr B_1$,
\item $\sp(B_1) \subseteq \left\{\left.\frac{k}{n} \, \right| \, k \in\{0, 1, \ldots, n\}\right\}$, and
\item $\lambda(B_1; 1) = 0 = \lambda(A_0; 1)$.
\end{enumerate}
To do this, we desire to change any occurrences of $1$ long the diagonal of $B_0$ into $\frac{n-1}{n}$.

Notice it must be the case that $\lambda\left(B_0; \frac{j}{n}\right) \neq 0$ for some $j \in \{0, 1, \ldots, n-2\}$.  Indeed otherwise we would have $\sp(B_{0}) = \left\{1, \frac{n-1}{n}\right\}$.  However, since $A_0 \prec_\tr B_{0}$ implies $\Tr(A_0) = \Tr(B_0)$, the conditions $\lambda(A_0; 1) = 0$ and $\sp(A_0) \subseteq \left\{\left.\frac{k}{n} \, \right| \, k \in\{0, 1, \ldots, n\}\right\}$ would contradict $\Tr(A_0) = \Tr(B_0)$.  Hence
\[
k = \max\left\{j \in \{0, 1, \ldots, n-2\} \, \left| \, \lambda\left(B_0; \frac{j}{n}\right) \neq 0 \right.\right\}
\]
is well-defined.

Let
\[
p = \min\left\{\lambda(B_0; 1), (n-1-k)\lambda\left(B_0; \frac{k}{n}\right)\right\}.
\]
We will now describe a process for changing the first $p$ 1s on the diagonal of $B_0$.

For convenience, let $x = \lambda(B_0; 1) + \lambda\left(B_0; \frac{n-1}{n}\right)$ and let $y = \lambda(B_0; 1)$. Note $p \leq x$.  Assume that $\alpha_1 \geq \alpha_2 \geq \cdots \geq \alpha_m$ and $\beta_1 \geq \beta_2 \geq \cdots \geq \beta_m$ are the eigenvalues arranged in non-increasing order of $A_0$ and $B_0$ respectively.  Then
\[
(\alpha_1, \ldots, \alpha_m) \prec (\beta_1, \ldots, \beta_m).
\]
We claim that
\[
(\alpha_1, \ldots, \alpha_m) \prec \left( \frac{n-1}{n}, \beta_2, \ldots, \beta_{x}, \frac{k+1}{n}, \beta_{x+2},\ldots, \beta_m\right).
\]
To see this, we notice that $\alpha_q \leq \frac{n-1}{n}$ for all $q$ so as $\beta_2, \ldots, \beta_x \geq \frac{n-1}{n}$, checking the sums for majorization works upto and including index $x$.  Since summing the first $x+1$ elements of the right-hand tuple produces $\sum^{x+1}_{j=1} \beta_j$, the majorization is then assured.   If $k+1 \neq n-1$, then we can apply the same argument to show that 
\[
(\alpha_1, \ldots, \alpha_m) \prec \left( \frac{n-1}{n}, \frac{n-1}{n}, \beta_3, \ldots, \beta_{x}, \frac{k+2}{n}, \beta_{x+2},\ldots, \beta_m\right).
\]
By continuing this process, we eventually change the first $y$ entries of the right-hand tuple to be $\frac{n-1}{n}$ (only occurs in the case that $p =\lambda(B_0; 1)$), or the $x+1^{\mathrm{st}}$ entry of the right-hand tuple becomes $\frac{n-1}{n}$ after $n-1-k$ progressions.  Using the same arguments, we may begin to change the $x+2^{\mathrm{nd}}$ entry in the same manner.  Repeating ad nauseum, we eventually obtain
\[
(\alpha_1, \ldots, \alpha_m) \prec \left( \frac{n-1}{n}, \ldots, \frac{n-1}{n}, \beta_{p+1}, \ldots, \beta_{x}, \frac{n-1}{n}, \ldots, \frac{n-1}{n}, \beta'_z,\ldots, \beta_m\right).
\] 
where $\beta'_z = \beta_z$ if $\beta_z \neq \frac{k}{n}$ and otherwise $\beta'_z$ is some element of $\left\{\frac{k}{n}, \ldots, \frac{n-1}{n}\right\}$.  Hence, if
\[
B_{0,1} = \diag \left( \frac{n-1}{n}, \ldots, \frac{n-1}{n}, \beta_{p+1}, \ldots, \beta_{x}, \frac{n-1}{n}, \ldots, \frac{n-1}{n}, \beta'_z,\ldots, \beta_m\right)
\]
then we have $A_0 \prec_\tr B_{0,1} \prec_\tr B_0$ by construction (i.e. to see $B_{0,1} \prec_\tr B_0$, apply Lemma \ref{lem:a-convex-combination-of-2} at each stage of the above process) and $\sp(B_{0,1}) \subseteq \left\{\left.\frac{k}{n} \, \right| \, k \in\{0, 1, \ldots, n\}\right\}$.

To count the number unitary conjugates of $B_0$ required in a convex combination to produce $B_{0,1}$, notice in the above process that the $x+1^{\mathrm{st}}$ entry of the right-hand tuple only interacts with the first $n-1-k$ entries of the right-hand tuple, the $x+2^{\mathrm{nd}}$ entry of the right-hand tuple only interacts with the next $n-1-k$ entries of the right-hand tuple, and so on.  Hence, using direct sums and Lemmata \ref{lem:a-convex-combination-of-2} and \ref{lem:progressive-convex-combinations}, we obtain that $B_{0,1}$ can be constructed as a convex combination of at most $2^{n-1-k}$ unitary conjugates of $B_0$.  

If $p = \lambda(B_0; 1)$, then we have successfully turned all of the
$1$ diagonal entries in $B_0$ into $\frac{n-1}{n}$s via $B_{0,1}$ (so let $B_1 = B_{0,1}$).  Otherwise $\lambda(B_{0,1}; 1) > 0$ and
\[
k > \max\left\{j \in \{0, 1, \ldots, n-2\} \, \left| \, \lambda\left(B_{0,1}; \frac{j}{n}\right) \neq 0 \right.\right\} \neq -\infty.
\]
Consequently, we can apply this process again, each time decreasing the value of $k$.  As there are only a finite number of values of $k$, this process must stop.  When the process terminates, the resulting $B_{0,l}$ must have the property that $\lambda(B_{0,l}; 1) = 0$ for otherwise if $\lambda(B_{0,l}; 1) > 0$ the process yields $\sp(B_{0,l}) = \left\{1, \frac{n-1}{n}\right\}$ which contradicts the fact that $A_0 \prec_\tr B_{0,l}$ via comparing traces as $\lambda(A_0; 1) = 0$ and $\sp(A_0) \subseteq \left\{\left.\frac{k}{n} \, \right| \, k \in\{0, 1, \ldots, n\}\right\}.$

Analyzing the above process, we see that we need at most
\[
\prod^{n-2}_{k=0} 2^{n-1-k} = 2^{\sum^{n-2}_{k=0} n-1-k} = 2^{\sum^{n-1}_{j=1} j} = 2^{\frac{(n-1)n}{2}}
\]
unitary conjugates in a convex combination to change $B_0$ into $B_1$ (i.e. the worst-case scenario is we never have $p = \lambda(B_0; 1)$ and we have to exhaust our way all the way down).

By repeating the above process, we may obtain $B_2$ via a convex combination of at most
\[
\prod^{n-3}_{k=0} 2^{n-2-k} = 2^{\sum^{n-3}_{k=0} n-2-k} = 2^{\sum^{n-2}_{j=1} j} = 2^{\frac{(n-2)(n-1)}{2}}
\]
unitary conjugates of $B_1$ such that
\begin{enumerate}
\item $A_0 \prec_\tr B_2$,
\item $B_2 \subseteq \left\{\left.\frac{k}{n} \, \right| \, k \in\{0, 1, \ldots, n\}\right\}$, and
\item $\lambda(B_2; 1) = \lambda(A_0; 1)$, and $\lambda\left(B_2; \frac{n-1}{n}\right) = \lambda\left(A_0; \frac{n-1}{n}\right)$.
\end{enumerate}
This process eventually stops at $B_n$ as we will have
\[
\lambda\left(B_n; \frac{n-k}{n}\right) = \lambda\left(A_0; \frac{n-k}{n}\right)
\]
for all $k \leq n-1$ so $A_0 \prec_\tr B_n$ will then imply $\lambda(B_n; 0) = \lambda(A_0; 0)$.  Consequently, the above process shows that $A_0$ is a convex combination of at most
\[
2^{\sum^{n}_{k=1}  \frac{(n-k)(n+1-k)}{2} } = \cdots = 2^{\frac{1}{6} (n^3 - n)}
\]
unitary conjugates of $B_0$.
\end{proof}

\begin{cor}
There exists constant $C_1, C_2 > 0$ for any unital, separable, simple, non-elementary C$^*$-algebra  $\A$ with non-trivial tracial simple, real rank zero, strict comparison of projections with respect to tracial states, and cancellation of projections, for any $n \in \bN$, and for any $a, b \in \A_{\sa}$ such that $a \prec_\tau b$ for all $\tau \in T(\A)$ there exists  $\ell = C_1 2^{\frac{1}{3}(n^3 - n)}$ unitaries $u_1, \ldots, u_{\ell} \in \A$ and $t_1, \ldots, t_\ell \in [0,1]$ such that $\sum^\ell_{k=1} t_k = 1$ and
\[
\left\|a - \sum^\ell_{k=1} t_k u_k^* b u_k\right\| < \frac{C_2}{n-1}.
\]
\end{cor}
\begin{proof}
Consider the proof of Theorem \ref{thm:main}.  Up to elementary perturbations, we obtain that $a \in \cconv(\U(b))$ by only constructing convex combinations of unitary conjugates twice; one via  Lemma \ref{lem:NormPerturbation} and once via Lemma \ref{lem:strict-comparison-convex-hull-provided-room}.  As the proofs of both Lemma \ref{lem:NormPerturbation} and Lemma \ref{lem:strict-comparison-convex-hull-provided-room} us only matricial results, one apply Theorem \ref{thm:bounding-number-in-approx-convex-combination} in their proofs to obtain the desired bounds.
\end{proof}

\section*{Acknowledgements}

The first author thanks Leonel Robert for pointing out the idea of uniform bounds and its relationship to regularity.

\end{document}